\documentclass[a4paper,10pt]{amsart}

\usepackage{amssymb,amsmath,amsfonts,amsthm,mathtools}
\usepackage{latexsym,mathrsfs,pb-diagram}
\usepackage[pdftex]{graphicx}
\usepackage[all]{xy}
\usepackage[hmargin=2.5cm,vmargin=3.5cm]{geometry}
\usepackage[plainpages=false,colorlinks,pdfpagelabels]{hyperref}
\hypersetup{
  urlcolor=black,
  citecolor=green,
  linkcolor=blue
}

\numberwithin{equation}{section}

\theoremstyle{definition}
\newtheorem{defn}{Definition}[section]

\theoremstyle{remark}

\newtheorem{rmk}[defn]{Remark}

\theoremstyle{plain}
\newtheorem{prop}[defn]{Proposition}
\newtheorem{cor}[defn]{Corollary}
\newtheorem{thm}[defn]{Theorem}
\newtheorem{lem}[defn]{Lemma}

\renewcommand{\Re}{\mathrm{Re}}
\renewcommand{\Im}{\mathrm{Im}}
\renewcommand{\epsilon}{\varepsilon}

\title[Propagation of singularities for non homogeneous systems]{An FBI characterization for Gevrey vectors on hypo-analytic structures and propagation of Gevrey singularities}
\author{N. Braun Rodrigues }
\date{March 2020}

\begin{document}

\maketitle

\begin{abstract}
    In this work we prove a FBI characterization for Gevrey vectors on hypo-analytic structures, and we analyze the main differences of Gevrey regularity and hypo-analyticity concerning the FBI transform. We end with an application of this characterization on a propagation of Gevrey singularities result, for solutions of the non-homogeneous system associated with the hypo-analytic structure, for analytic structures of tube type. 
\end{abstract}

\section{Introduction}

In 1983 N. Hanges and F. Treves proved in \cite{hanges92} that on CR (embedded) manifolds, holomorphic extendability for CR functions propagates along connected complex submanifolds. They actually proved their result on the set up of hypo-analytic structures, introduced by  M.S. Baouendi, C.H. Chang, and F. Treves in \cite{baouendi1983microlocal}, and they proved that hypo-analytic singularity of solutions propagates along connected elliptic submanifolds. Propagation of holomorphic extendability is widely studied in the context of CR geometry, for instance \cite{trepreau1990propagation}, \cite{baouendi1984holomorphic} and \cite{tumanov1997propagation}. Now for Gevrey regularity little is known concerning propagation of singularities on hypo-analytic structures. In 2000 P. Caetano started the study of Gevrey vectors on hypo-analytic structures of maximum codimension in his Ph.D dissertation (\cite{caetano2000classes}), and his work was continued in \cite{caetano2011gevrey} and \cite{ragognette2019ultradifferential}, but their aim was solvability questions for the associate differential complex. Our goal here is to initiate the study of regularity problems on these structures, for instance, propagation of Gevrey singularities. 

A very useful tool in the study of propagation of singularities is the FBI transform. Also in \cite{baouendi1983microlocal} the authors proved that the decay of the FBI transform can be used to characterize hypo-analyticity. The usual characterization of analytic regularity, Gevrey regularity (ultradifferential regularity) and smooth regularity of distributions on $\mathbb{R}^N$ by the decay of the FBI transform differs from one another by the type of their decay. Loosely speaking, a distribution $u$ is analytic at $x_0$ if 

\begin{equation*}
    |\mathfrak{F}[\chi u](x,\xi)|\leq Ce^{-\epsilon|\xi|},
\end{equation*}

\noindent for all $x$ in some neigborhood of $x_0$, all $\xi\in\mathbb{R}^N$, and for some positive constants $C,\epsilon$, where $\chi$ is a test function supported in some open neighborhood of $x_0$, and $\mathfrak{F}[\chi u](x,\xi)$ is the FBI transform of $\chi u$. Now $u$ is Gevrey at $x_0$ if 

\begin{equation*}
    |\mathfrak{F}[\chi u](x,\xi)|\leq Ce^{-\epsilon|\xi|^{\frac{1}{s}}},
\end{equation*}

\noindent for all $x$ in some neighborhood of $x_0$, all $\xi\in\mathbb{R}^N$, and for some positive constants $C,\epsilon$. So the difference between analytic regularity and Gevrey regularity in this context is the type of the bound. On hypo-analytic structures there is an additional difficulty that arises from its complex nature which remains unseen when dealing with analytic regularity\footnote{In the context of hypo-analytic structures, by analytic regularity we mean hypo-analyticity}.  

For simplicity let $M\subset\mathbb{C}^N$ be a smooth generic CR submanifold of codimension $d$, so the CR dimension of $M$ is $n=N-d$, and let $p$ be an arbitrary point of $M$. Therefore there are $\mathrm{L}_1,\dots,\mathrm{L}_n$ anti-holomorphic vector fields tangent to $M$ on a neighborhood of $p$, and real vector fields $\mathrm{T}_1,\dots,\mathrm{T}_d$ tangent to $M$ on a neighborhood of $p$ such that $\{\mathrm{L}_1,\dots,\mathrm{L}_n,\overline{\mathrm{L}_1},\dots,\overline{\mathrm{L}_n},\mathrm{T}_1,\dots,\mathrm{T}_d\}$ span the complexified tangent bundle of $M$ on a neighborhood of $p$. In this set up our main theorem states that for a CR function $u$ to be a Gevrey vector for $\overline{\mathrm{L}_1},\dots,\overline{\mathrm{L}_n},\mathrm{T}_1,\dots,\mathrm{T}_d$ it is necessary and sufficient that its FBI transform has the same bound as in the $\mathbb{R}^N$ scenario, but only for points on the so-called real structure bundle, which is a real subbundle of $(\mathrm{T}^{0,1}M)^\perp$ (cf. section $2.3.$). Here one might notice that we are not asking any additional regularity on the CR structure.

Our initial goal was to investigate the validity of the propagation of singularities result proved in \cite{hanges92} for Gevrey regularity. One difficulty is that this result is deeply based on holomorphic function theory, which is not available for the Gevrey case. One of the drawbacks, when using the same techniques (the FBI approach), is that in our case we need some sort of folliation near the "propagator", an unnecessary assumption on \cite{hanges92}. On the other hand, this technique allows us to consider solutions of the non-homogeneous system, which makes sense in the Gevrey scenario.

This paper is organized as follows: In the first section we discuss what is needed from locally integrable structures theory, and hypo-analytic structures theory. Then in section \ref{section:Gevrey} we prove a FBI characterization for Gevery vectors, and in the last section we prove (and give some examples) a propagation of Gevrey singularities result on hypo-analytic structures of tube type.

This work contains the results obtained by the author on his Ph.D. dissertation.  

\section{Preliminaries}

\subsection{Locally integrable structures}

Let $\Omega\subset\mathbb{R}^N$ be an open set. By a locally integrable structure on $\Omega$ we mean a complex vector bundle $\mathcal{V}\subset\mathbb{C}\mathrm{T}\Omega$, such that $[\mathcal{V},\mathcal{V}]\subset\mathcal{V}$, and at every point $p\in\Omega$ there are $Z_1,\dots,Z_m$, smooth, complex-valued functions in some open neighborhood of $p$ in $\Omega$, such that
\begin{equation*}
\begin{cases}
    \mathrm{d}Z_1\wedge\cdots\wedge\mathrm{d}Z_m\neq 0;\\
    \mathrm{L}Z_j=0,\quad\forall \,\mathrm{L}\in\mathcal{V},\;j=1,\dots,m.
\end{cases}
\end{equation*}

\noindent We denote by $\mathrm{T}^\prime\subset\mathbb{C}\mathrm{T}^\ast\Omega$ the orthogonal bundle, with respect to the duality between forms and vectors, of the bundle $\mathcal{V}$. Let $p$ be an arbitrary point at $\Omega$. Then there exist a local coordinate system vanishing at $p$ on some open set $U=V\times W$, $(x_1,\dots, x_m,t_1,\dots, t_n)$, and smooth, real-valued functions $\phi_1,\dots, \phi_m$, defined on $U$ and satisfying $\phi(0)=0$ and $\mathrm{d}_x\phi(0)=0$, such that the differentials of the functions

\begin{equation}\label{eq:Z}
    Z_k(x,t)\doteq x_k+i\phi_k(x,t),\quad k=1,\dots,m,
\end{equation}

\noindent span $\mathrm{T}^\prime$ in $U$. There are also linear independent, pairwise commuting, complex vector fields:

\begin{equation*}
    \mathrm{M}_j=\sum_{k=1}^m a_{j,k}(x,t)\frac{\partial}{\partial x_k},\quad j=1,\dots, m,
\end{equation*}

\noindent and

\begin{equation*}
    \mathrm{L}_j=\frac{\partial}{\partial t_j}-i\sum_{k=1}^m\frac{\partial\phi_k}{\partial t_j}(x,t)\mathrm{M}_k,\quad j=1,\dots,n,
\end{equation*}

\noindent satisfying the relations

\begin{center}
\begin{tabular}{ll}
    $\mathrm{L}_j Z_k=0$  & $\mathrm{M}_l Z_k = \delta_{l,k}$\\
   $\mathrm{L}_j t_i = \delta_{j,i}$  & $\mathrm{M}_l t_i=0$.
\end{tabular}
\end{center}

\noindent  Now let $u$ be a distribution on $U$ such that $\mathrm{L}_ju\in\mathcal{C}^\infty(U)$, for $j=1,\dots, n$, then actually $u\in\mathcal{C}^\infty(W;\mathcal{D}^\prime(V))$ (see proof of Proposition I.$4.3$ of \cite{trevesbook} with minor modifications). By an uniform boundedness principle argument we have that for every compact sets $K_1\Subset V$, and $K_2\Subset W$, there exist a constant $C>0$ and an integer $q>0$ such that

\begin{equation}\label{q:uniform-cont-distribution-solution}
    \left|\langle u(x,t), \phi(x) \rangle\right|\leq C\sum_{|\alpha|\leq q}\sup_{x\in K_1}\left|\partial^\alpha\phi\right|,\quad\forall\phi\in\mathcal{C}^\infty_c(K_1),
\end{equation}

\noindent for every $t\in K_2$. 

Now let $\mathcal{H}\subset\Omega$ be a (embedded) submanifold. We say that $\mathcal{H}$ is maximally real if 

\begin{equation*}
    \mathbb{C}\mathrm{T}_p\Omega=\mathcal{V}_p\oplus\mathbb{C}\mathrm{T}_p\mathcal{H},\quad\forall p\in\mathcal{H},
\end{equation*}

\noindent or equivalently,

\begin{equation*}
    \mathbb{C}\mathrm{T}^\ast_p\Omega=\mathbb{C}\mathrm{N}^\ast_p\mathcal{H}\oplus\mathrm{T}^\prime_p,\quad\forall p\in\mathcal{H}.
\end{equation*}

\subsection{Hypo-analytic structures}

 Let $\Omega\subset \mathbb{R}^N$ be an open set. A hypo-analytic structure on $\Omega$ is a pair $\{(U_\alpha)_{\alpha\in\Lambda},(Z_\alpha)_{\alpha\in\Lambda}\}$ such that 
  
  \begin{itemize}
      \item $(U_\alpha)_{\alpha\in\Lambda}$ is an open covering for $\Omega$;
      \item $Z_\alpha:U_\alpha\longrightarrow\mathbb{C}^m$ is a smooth map, for every $\alpha\in\Lambda$;
      \item $\mathrm{d}Z_{\alpha,1},\dots,\mathrm{d}Z_{\alpha,m}$ are $\mathbb{C}-$linear independent on $U_\alpha$, for every $\alpha\in\Lambda$;
      \item if $\alpha\neq\beta$, then to each $p\in U_\alpha\cap U_\beta$ there is a holomorphic map $F$ such that $Z_\alpha=F\circ Z_\beta$, in a neighborhood of $p$ in $U_\alpha\cap U_\beta$;
      \item if $Z:U\longrightarrow\mathbb{C}^m$ is a smooth function such that for every $p\in U\cap U_\alpha$ there exists a holomorphic function $F$ such that $Z=F\circ Z_\alpha$, then $(U,Z)=(U_\beta,Z_\beta)$, for some $\beta\in\Lambda$.  
  \end{itemize}
 
 \noindent We call each pair $(U_\alpha,Z_\alpha)$ as a hypo-analytic chart. We say that a distribution $u\in\mathcal{D}^\prime(\Omega)$ is hypo-analytic at $p$ if for some $\alpha\in\Lambda$ such that $p\in U_\alpha$, there is a holomorphic function $F$, defined on a complex neighborhood of $Z_\alpha(p)$, such that $u=F\circ Z_\alpha$, in some open neighborhood of $p$. Given a hypo-analytic structure on $\Omega$ we can associate a locally integrable structure $\mathcal{V}$ setting its orthogonal $\mathrm{T}^\prime$ as the complex bundle locally defined by the differentials $\mathrm{d}Z_1,\dots,\mathrm{d}Z_m$. So let $p\in\Omega$ and $(U,Z)$ a hypo-analytic chart, with $p\in U$. We can assume that there are local coordinates $(x_1,\dots, x_m,t_1, \dots, t_n)$ in $U=V\times W$, as described in the section above, so the function $Z$ is given by \eqref{eq:Z}. Note that in this coordinate system, the point $p$ is the origin.
 
 \begin{defn}
 Let $s>1$. We say that a distribution $u$ on $U$ is a Gevrey-$s$ vector if $u$ is a smooth function on $U$, and for every compact set $K\subset U$ there exists a constant $C>0$ such that

\begin{equation*}
    \sup_{(x,t)\in K}|\mathrm{L}^\alpha\mathrm{M}^\beta u(x,t)|\leq C^{|\alpha|+|\beta|+1}\alpha!^s\beta!^s,\quad\forall \alpha\in\mathbb{Z}_+^n, \beta\in\mathbb{Z}_+^m.
\end{equation*}

\noindent We denote by $\mathrm{G}^s(U; \mathrm{L}_1,\cdots,\mathrm{L}_n,\mathrm{M}_1,\cdots,\mathrm{M}_m)$ the space of all Gevrey-$s$  vectors on $U_1$. If $s=1$ we say that $u$ is an analytic vector, and we write $u\in\mathcal{C}^\omega(U; \mathrm{L}_1,\cdots,\mathrm{L}_n,\mathrm{M}_1,\cdots,\mathrm{M}_m))$.
 \end{defn}
 
 \noindent We have the following characterization of Gevrey vectors in terms of almost-analytic extensions:
 
 \begin{thm}[Theorem $1.1$ of \cite{caetano2000classes}]
Let $u$ be a distribution on $U$ and $U_1=V_1\times W_1\Subset U$, where $V_1$ and $W_1$ are balls centered at the origin. Are equivalent:

\begin{enumerate}
    \item $u$ is a Gevrey-$s$ vector on $U_1$;
    \item There exist $\mathcal{O}$ an open neighborhood of $(Z(U_1),W_1)$ on $\mathbb{C}^{n+m}$ and a Gevrey function $F\in\mathrm{G}^s(\mathcal{O})$ such that
    
    \begin{equation*}
        \begin{cases}
            F(Z(x,t),t)=u(x,t),\quad\forall (x,t)\in U_1;\\
            \big(\overline{\partial_z}+\overline{\partial_\tau}\big)F(z,\tau)\sim 0,\quad \text{on}\;\; (Z(U_1),W_1).
        \end{cases}
    \end{equation*}
\end{enumerate}
\end{thm}

\noindent Here $f\sim 0$ at $\Sigma$ means that $f$ is flat on $\Sigma$. A useful consequence of this theorem, that we shall use later on, is the following:

\begin{cor}\label{cor:caetano-decay}
Let $u$ be a distribution on $U$ and $U_1=V_1\times W_1\Subset U$, where $V_1$ and $W_1$ are open balls centered at the origin. Suppose that $u|_{U_1}\in\mathrm{G}^s(U_1;\mathrm{L}_1,\dots\mathrm{L}_n,\mathrm{M}_1,\dots,\mathrm{M}_m)$. Then there are an open neighborhood $\mathcal{O}$ of $\{Z(x,t)\;:\;x\in V_1,t\in W_1\}$ on $\mathbb{C}^m$ and a smooth function $F\in \mathcal{C}^\infty(\mathcal{O}\times W_1)$ such that
        
       \begin{equation}\label{eq:caetano-decay}
            \begin{cases}
                u(x,t)=F(Z(x,t),t),\quad  (x,t)\in V_1\times W_1\\
                \big|\partial_{\overline{z}}F(z,t)\big|\leq C^{k+1}k!^{s-1}\mathrm{dist}(z;\mathfrak{W}_t)^k,\quad \forall k\in\mathbb{Z}_+, z\in\mathcal{O}, t\in W_1,
            \end{cases}
        \end{equation}
        
\noindent where $C$ is a positive constant, and 

\begin{equation*}
    \mathfrak{W}_t=\{Z(x,t)\;:\;x\in V_1\}.
\end{equation*}

\end{cor}

\noindent This corollary is a consequence of previous theorem and the Taylor formula.
Despite the difference between Gevrey and analytic vectors beeing a power of $s$ in their definition, they have very different properties. To illustrate this difference let us recall some well-known properties of hypo-analytic functions: 

Let $u$ be a distribution on $U$ such that $\mathrm{L}_j u=0$, $j=1,\dots,n$. Then it is equivalent (see \cite{trevesbook}):

\begin{enumerate}
    \item $u$ is hypo-analytic at the origin;
    \item the restriction of $u$ to a maximally real submaniold, passing through the origin, is hypo-analytic at the origin (with respect to the induced hypo-analytic structure);
    \item $u$ is an analytic vector in some open neighborhood of the origin.
\end{enumerate}

\noindent Let us prove that $(2)\Rightarrow(1)$. So let $\mathcal{H}$ be a maximally real submanifold such that $u|_\mathcal{H}$ is hypo-analytic at $p$. Then there exists $U_\mathcal{H}$ an open neighborhood of $p$ on $\mathcal{H}$, and a holomorphic function $F$ defined on $\mathcal{O}$, an open neighborhood  of $Z(U_\mathcal{H})$ on $\mathbb{C}^m$, such that 

\begin{equation*}
    u(p^\prime)=F(Z(p^\prime)),\quad\forall p^\prime\in U_\mathcal{H}.
\end{equation*}

\noindent Now set $\tilde{u}\doteq F\circ Z$, defined in some neighborhood of $p$ on $\Omega$. Since $F$ is holomorphic we have that $L \tilde{u}=0$, for every $L\in\mathcal{V}$, so the same is valid for $u-\tilde{u}$, and $u-\tilde{u}$ vanishes on a neighborhood of $p$ on $\mathcal{H}$. By a standard uniqueness result, based on the Baouendi-Treves approximation formula, we have that $u-\tilde{u}$ vanishes on some neighborhood of $p$, \textit{i.e.}, $u$ is hypo-analytic at $p$. Note that a key ingredient of this argument is that the composition of a holomorphic function with the first integrals $Z$s are solutions in a full neighborhood of $p$. So in the Gevrey scenario, where the function $F$ would be a Gevrey function such that $\overline{\partial_{z}}F$ is flat on $Z(U_\mathcal{H})$, we do not have this same phenomena anymore, that is, $F\circ Z$ is not a solution on a full neighborhood of $p$, just on $U_\mathcal{H}$, so we cannot apply the uniqueness result in this case. Conclusion: For Gevrey regularity, testing on maximally real submanifolds is not enough.
 
\subsection{The real structure bundle and the FBI transform}

An object that plays a central role in the analysis on hypo-analytic structures is the so-called real structure bundle. It allows us to mimic some "real techniques" on this complex scenario. Let $\mathcal{H}\subset\mathbb{C}^m$ be a maximally real submanifold (\textit{i.e.}, the restriction of the coordinate functions $z_1,\dots, z_m$ to $\mathcal{H}$ defines an hypo-analytic structure of co-rank $0$). Suppose that the origin belongs to $\mathcal{H}$, so $\mathcal{H}$ is locally the image of the map

\begin{equation*}
    Z(x)=x+i\phi(x),
\end{equation*}

\noindent where the function $\phi$ is real-valued, $\phi(0)=0$ and $\mathrm{d}\phi(0)=0$. The real structure bundle of $\mathcal{H}$ is locally defined as 

\begin{equation*}
    \mathbb{R}\mathrm{T}^\prime_{\mathcal{H}}=\{(Z(x),{}^\mathrm{t}Z_x(x)^{-1}\xi)\;:\; x\in U,\,\xi\in\mathbb{R}^m\},
\end{equation*}

\noindent where $U$ is the open neighborhood of the origin where the map $Z$ is defined. For every $\kappa>0$ we write

\begin{equation*}
    \mathfrak{C}_\kappa\doteq\{\zeta\in\mathbb{C}^m\;:\;|\Im\zeta|<\kappa|\Re\zeta|\}.
\end{equation*}

\noindent If $\zeta\in\mathbb{C}^m$ we write $\langle\zeta\rangle^2\doteq \zeta\cdot\zeta=\zeta_1^2+\cdots+\zeta_m^2$. Using the main branch of the square root we can define $\langle\zeta\rangle=[\langle\zeta\rangle^2]^{1/2}$, for $\zeta\in\mathfrak{C}_\kappa$, .

\begin{defn}
We shall say that the maximally real submanifold $\mathcal{H}$ of $\mathbb{C}^m$ at one of its points, $p$, is well positioned if there is a number $\kappa$, $0<\kappa<1$, and an open neighborhood $U$ of $p$ in $\mathcal{H}$ such that

\begin{equation*}
    \forall z,z^\prime\in U, \zeta\in \mathbb{R}^m,\;\textrm{and}\;\zeta\in\left(\left.\mathbb{R}\mathrm{T}^\prime_{\mathcal{H}}\right|_z\right)\cap\left(\left.\mathbb{R}\mathrm{T}^\prime_{\mathcal{H}}\right|_{z^\prime}\right)\;\textrm{then},
\end{equation*}
 \begin{equation}\label{eq:maximally-real-well-positioned}
     \begin{cases}
     |\Im\zeta|<\kappa|\Re\zeta|;\\
       \Im\left\{\zeta\cdot(z-z^\prime)+i\langle\zeta\rangle\langle z-z^\prime\rangle^2\right\}\geq (1-\kappa)|\zeta||z-z^\prime|^2.
     \end{cases}
 \end{equation}
 
 We shall say that $\mathcal{H}$ is very well-positioned at $p$ if, given any $0<\kappa<1$, there is an open neighborhood $U$ of $p$ in $\mathcal{H}$ such that \eqref{eq:maximally-real-well-positioned} is valid. 
\end{defn}

\noindent After applying a biholomorphism we can always assume that a maximally real submanifold is very well positioned at $p$. Now we recall the definition of the FBI transform on maximally real submanifolds:

\begin{defn}
Let $u\in\mathcal{E}^\prime(\mathcal{H})$. So we define

\begin{equation*}
    \mathfrak{F}[u](z,\zeta)\doteq\left\langle u(z^\prime), e^{i\zeta\cdot(z-z^\prime)-\langle\zeta\rangle\langle z-z^\prime\rangle^2}\Delta(z-z^\prime,\zeta)\right\rangle,
\end{equation*}

\noindent for $z\in\mathbb{C}^m$ and $\zeta\in\mathfrak{C}_1$, where

\begin{equation*}
    \Delta(z,\zeta)=\det(\mathrm{Id}+i(z\odot \zeta)/\langle\zeta\rangle),
\end{equation*}

\noindent and $(z\odot\zeta)=(z_i\zeta_j)_{i,j=1,\cdots,m}$.

\end{defn}

\noindent The real structure bundle is essential when it comes to estimates, as on can see in the next proposition (Proposition IX$.1.1.$ and Proposition IX$.2.1.$ of \cite{trevesbook}):

\begin{prop}\label{prop:basic-estimates-maximally-real-FBI}
Let $u\in\mathcal{E}^\prime(\mathcal{H})$. Then $\mathfrak{F}[u](z,\zeta)$ is holomorphic in $(z,\zeta)\in\mathbb{C}^m\times\mathfrak{C}_1$ and for every $K\subset \mathbb{C}^m$ compact set and every $0<\kappa<1$, there are constants $C,R>0$ such that

\begin{equation*}
    |\mathfrak{F}[u](z,\zeta)|\leq Ce^{R|\zeta|},\quad \forall z\in K,\,\zeta\in\mathfrak{C}_\kappa.
\end{equation*}

\noindent Now if in addition $\mathcal{H}$ is well positioned at $p\in\mathcal{H}$, then there exists an open neighborhood $U$ of $p$ on $\mathcal{H}$ such that if the support of $u$ is contained in $U$, there exist an integer $k>0$ and a constant $C>0$ such that

\begin{equation*}
    |\mathfrak{F}[u](z,\zeta)|\leq C(1+|\zeta|)^k,\quad (z,\zeta)\in\left.\mathbb{R}\mathrm{T}^\prime_{\mathcal{H}}\right|_{U},
\end{equation*}

\noindent where $(z,\zeta)\in\left.\mathbb{R}\mathrm{T}^\prime_{\mathcal{H}}\right|_{U}$ means that $z\in U$ and $\zeta\in \left.\mathbb{R}\mathrm{T}^\prime_{\mathcal{H}}\right|_{z}$.

\end{prop}

\noindent The FBI transform can be used to characterize holomorphic extendability, as well as other kinds of extandabilities, as we shall see later on. To prove holomorphic extendability using the FBI transform one needs the following inversion formula:

\begin{prop}[Proposition IX$.2.2.$ of \cite{trevesbook}]
    Let $u\in\mathcal{E}^\prime(\mathcal{H})$. Then
    
    \begin{equation}\label{eq:first-inversion-formula}
        u(z)=\lim_{\epsilon\to 0^+}\frac{1}{(2\pi)^m}\int_{\mathrm{R}^m} e^{-\epsilon|\xi|^2}\mathfrak{F}[u](z,\xi)\mathrm{d}\xi,
    \end{equation}
    
    \noindent where the convergence is in $\mathcal{D}^\prime(\mathcal{H})$.
    \end{prop}
    
\noindent So with this inversion formula one can prove the following theorem:

 \begin{thm}[Theorem IX$.3.1.$ of \cite{trevesbook}]
    Let $u\in\mathcal{E}^\prime(\mathcal{H})$ and $p\in\mathcal{H}$. The following are equivalent:
    \begin{enumerate}
        \item There exists $\mathcal{O}\subset\mathbb{C}^m$, an open neighborhood of $p$, $F$ a holomorphic function at $\mathcal{O}$, such that $u|_{\mathcal{O}\cap \Sigma}=F|_{\mathcal{O}\cap \Sigma}$;\\
        \item There exists $\mathcal{O}\subset\mathbb{C}^m$, an open neighborhood of $p$, $0<\kappa^\prime<1$, and $C,\epsilon>0$ such that
        
        \begin{equation*}
            \left|\mathfrak{F}[u](z,\zeta)\right|\leq Ce^{-\epsilon|\zeta|},\quad\forall (z,\zeta)\in\mathcal{O}\times\mathfrak{C}_{\kappa^\prime}
        \end{equation*}
       \item There exists $\mathcal{O}\subset\mathbb{C}^m$, an open neighborhood of $p$, such that $\mathfrak{F}[u](z,\xi)$ is bounded by an integrable function with respect to $\xi\in\mathbb{R}^m$, uniformly in $z\in\mathcal{O}$.
    \end{enumerate}
    \end{thm}
    
\noindent The third equivalence of this theorem does not appear in the literature, but it is how actually one prove that $(2)$ implies $(1)$, using the inversion formula \eqref{eq:first-inversion-formula}. This simple observation illustrates the advantage of having holomorphic function theory at our disposal. In this cases the FBI decay is not very important because we have the control of it on a full neighborhood of $p$. Now if one wants to measure, for instance, smooth regularity with the FBI transform, then the estimate and where the estimate takes place are both very important.

 \begin{thm}[Theorem IX$.4.1$ of \cite{trevesbook}]
   Let $u\in\mathcal{E}^\prime(\mathcal{H})$ and $p\in\mathcal{H}$. Then are equivalent:
   \begin{enumerate}
       \item $u$ is $\mathcal{C}^\infty$ near $p$;\\
       \item There exists $U$ a neighborhood of $p$, such that for every $k\in \mathbb{Z}_+$ there is a $C_k>0$, such that
       
       \begin{equation*}
           \left|\mathfrak{F}[u](z,\zeta)\right|\leq C_k(1+|\zeta|)^{-k},\quad\forall(z,\zeta)\in\left.\mathbb{R}\mathrm{T}^\prime_{\mathcal{H}}\right|_{U}.
       \end{equation*}
   \end{enumerate}
    \end{thm}

\section{A FBI characterization of Gevrey vectors}\label{section:Gevrey}

Since the main result of this section, Theorem \ref{thm:gevrey-FBI-hypo}, is a local result, we will fix an arbitrary point at $\Omega$, and for simplicity we shall call it the origin. As we saw in the previous section, there is a hypo-analytic chart $(U,Z_1(x,t),\cdots,Z_m(x,t))$, with $0\in U$, and we can assume that the $Z$'s are given by

\begin{equation*}
    Z_j(x,t)=x_j+i\phi_j(x,t),\qquad j=1,\cdots,m,
\end{equation*}

\noindent with  $(x,t)\in U$, where the map $\phi(x,t)=(\phi_1(x,t),\cdots,\phi_m(x,t))$ is smooth, real-valued, $\phi(0)=0$ and $\mathrm{d}_x\phi(0)=0$. We can associate to it the complex vector fields $\{\mathrm{M}_1,\cdots,\mathrm{M}_m,\mathrm{L}_1,\cdots,\mathrm{L}_n\}$ with the following properties:

\begin{center}
    \begin{tabular}{l l}
        $\mathrm{L}_jZ_k=0$ & $\mathrm{M}_jZ_k=\delta_{j,k}$ \\
        $\mathrm{L}_jt_k=\delta_{j,k}$ & $\mathrm{M}_j t_k=0$.
    \end{tabular}
\end{center}

\noindent We can also assume that

\begin{equation}\label{eq:bound-phi}
    |\phi_j(x,t)|\leq C(|x|^3+|t|),\qquad j=1,\cdots,m,
\end{equation}

\noindent for some positive constant $C$, and 

\begin{equation}\label{eq:bound-derivative-phi}
    |\phi(x,t)-\phi(x^\prime,t)|\leq\mu|x-x^\prime|,
\end{equation}

\noindent with $0<\mu$ small as we want, for instance, less than $1$ (see pg. 433 of \cite{trevesbook}). 
From now on we are going to assume that $U=V\times W$, where $V\subset\mathbb{R}^m$ and $W\subset\mathbb{R}^n$ are balls centered at the origin.
Under this assumptions we can assume that for some $0<\kappa<1$ and $c>0$:

\begin{equation*}
    \forall x,x^\prime\in V, t\in W, \xi\in \mathbb{R}^m,\;\textrm{if}\;\zeta={}^\mathrm{t}Z_x(x,t)^{-1}\xi\;\textrm{then},
\end{equation*}
 \begin{equation}\label{eq:well-positioned}
     \begin{cases}
     |\Im\zeta|<\kappa|\Re\zeta|;\\
       \Im\left\{\zeta\cdot(Z(x,t)-Z(x^\prime,t))+i\langle\zeta\rangle\langle Z(x,t)-Z(x^\prime,t)\rangle^2\right\}\geq c|\zeta||Z(x,t)-Z(x^\prime,t)|^2.
     \end{cases}
 \end{equation}
  
\noindent For every $t\in W$ we define the maximally real submaniold $\mathfrak{W}_t$ as

\begin{equation*}
    \mathfrak{W}_t\doteq\{Z(x,t)\,:\,x\in V\}\subset \mathbb{C}^m,
\end{equation*}

\noindent and we can write the real structure bundle of $\mathfrak{W}_t$ as 

\begin{equation*}
    \mathbb{R}\mathrm{T}^\prime|_{\mathfrak{W}_t}=\{(Z(x,t),{}^tZ_x(x,t)^{-1}\xi\,:\, x\in V, \xi\in\mathbb{R}^m\setminus 0)\}.
\end{equation*}

\noindent We can also assume that 

\begin{equation}\label{eq:1/2-well-positioned}
     \Im\left\{\zeta\cdot(Z(x,t)-Z(x^\prime,t))+i\frac{1}{2}\langle\zeta\rangle\langle Z(x,t)-Z(x^\prime,t)\rangle^2\right\}\geq c|\zeta||Z(x,t)-Z(x^\prime,t)|^2,
\end{equation}

\noindent for every $x,x^\prime\in V$, $t\in W$, $\zeta\in\left.\mathbb{R}\mathrm{T}^\prime_{\mathfrak{W}_t}\right|_{Z(x,t)}\cup\left.\mathbb{R}\mathrm{T}^\prime_{\mathfrak{W}_t}\right|_{Z(x^\prime,t)}$. One consequence of \eqref{eq:well-positioned} is that for every $\zeta\in\zeta\in\left.\mathbb{R}\mathrm{T}^\prime_{\mathfrak{W}_t}\right|_{Z(x,t)}$ the following is valid:

\begin{equation}\label{eq:estimate-re-langle-zeta-rangle}
    \Re\langle\zeta\rangle\geq\sqrt{\frac{1-\kappa^2}{1+\kappa^2}}|\zeta|\quad\text{and}\quad \Im\langle\zeta\rangle\leq |\zeta|.
\end{equation}

\begin{defn}
Let $u\in\mathcal{C}^\infty(W;\mathcal{E}^\prime(V))$ and $\lambda>0$. We define the FBI transform of $u$ as

\begin{equation*}
    \mathfrak{F}^\lambda[u](t;z,\zeta)=\int_V e^{i\zeta\cdot(z-Z(x^\prime,t))-\lambda\langle\zeta\rangle\langle z-Z(x^\prime,t)\rangle^2}u(x^\prime,t)\Delta(\lambda(z-Z(x^\prime,t)),\zeta))\mathrm{d}Z(x^\prime,t),
\end{equation*}

\noindent with $z\in\mathbb{C}^m$ and $\zeta\in\mathfrak{C}_1\setminus 0$.

\end{defn}

\noindent If we denote by $\widetilde{u}(z,t)=u(x,t)$, for $z=Z(x,t)$, then we can write

\begin{equation*}
    \mathfrak{F}^\lambda[u](t;z,\zeta)=\int_{\mathfrak{W}_t}e^{i\zeta\cdot(z-z^\prime)-\lambda\langle\zeta\rangle\langle z-z^\prime\rangle^2}\widetilde{u}(z^\prime,t)\Delta(\lambda(z-z^\prime),\zeta)\mathrm{d}z^\prime.
\end{equation*}

\noindent Note that the integral is to be understood in the dual sense. Since $u$ has compact support in $x$ we have that $\mathfrak{F}^\lambda[u](t;z,\zeta)$ is holomorphic with respect to $(z,\zeta)\in\mathbb{C}^m\times\mathfrak{C}_1\setminus 0$, and $\mathcal{C}^\infty$ with respect to $t$. For simplicity we write $\mathfrak{F}[u](t;z,\zeta)$, for $\lambda=1$. As in \ref{prop:basic-estimates-maximally-real-FBI}, we have the following bound for the FBI transform:

\begin{lem}
Let $u\in\mathcal{C}^\infty(W,\mathcal{E}^\prime(V))$. Then there exist $C>0$ and $k\in\mathbb{Z}_+$ such that

\begin{equation}\label{eq:first-bound-FBI}
    |\mathfrak{F}[u](t;z,\zeta)|\leq C(1+|\zeta|)^k,\quad\forall(z,\zeta)\in\mathbb{R}\mathrm{T}^\prime_{\mathfrak{W}_t}.
\end{equation}

\end{lem}

\noindent Every characterization via control of the decay/growth of the FBI transform is based on an inversion formula. The one that we will use here is not quite the same as in \cite{trevesbook}, so we will present its proof. We start recalling the following inversion formula (usefull when dealing with holomorphic extendability, see \cite{trevesbook}):

\begin{prop}\label{prop:FBI-intermidiate-inversion-formula}
Let $u\in \mathcal{C}^\infty(W;\mathcal{E}^\prime(V))$. For every $\epsilon>0$ set 

\begin{equation}\label{eq:u_epsilon-FBI-inversion-intermediate}
    u_\epsilon(x,t)\doteq \frac{1}{(2\pi)^{m}}\int_{\left.\mathbb{R}\mathrm{T}^\prime_{\mathfrak{W}_t}\right|_{Z(x,t)}}e^{-\epsilon\langle\zeta\rangle^2}\mathfrak{F}^\frac{1}{2}[u](t;Z(x,t),\zeta)\mathrm{d}\zeta.
\end{equation}

\noindent Then $u_\epsilon(x,t)\rightarrow u(x,t)$ in $\mathcal{C}^\infty(W;\mathcal{D}^\prime(V))$.

\end{prop}

\begin{rmk}
The integral \eqref{eq:u_epsilon-FBI-inversion-intermediate} is to be interpreted as 

\begin{equation*}
    \lim_{\delta\to 0^+}\int_{\{\zeta\in\left.\mathbb{R}\mathrm{T}^\prime_{\mathfrak{W}_t}\right|_{Z(x,t)}\;:\;|\zeta|>\delta\}}e^{-\epsilon\langle\zeta\rangle^2}\mathfrak{F}^\frac{1}{2}[u](t;Z(x,t),\zeta)\mathrm{d}\zeta.
\end{equation*}

\end{rmk}

\noindent We shall use this inversion formula to prove the one that we will actually use. But before doing so we need to extend the function $Z(x,t)$ with respect to the variable $x$ to the whole $\mathbb{R}^m$:\\

Let $V_1\Subset V$ and let $\psi\in\mathcal{C}^\infty_c(V)$ satisfying

\begin{align*}
    &0\leq\psi\leq 1\\
    &\psi\equiv 1\quad \text{in}\;V_1.
\end{align*}

\noindent Define $\widetilde{Z}(x,t)\doteq x+i\psi(x)\phi(x,t)$. Then $\widetilde{Z}$ defines the hypo-analytic structure in $V_1\times W$, but $\widetilde{Z}(x,t)$ is defined for all $x\in\mathbb{R}^m$. Also note that ${}^t\widetilde{Z}_x(x,t)^{-1}=(\mathrm{Id}-i{}^t(\psi\phi)_x(x,t))(\mathrm{Id}+{}^t(\psi\phi)_x(x,t)^2)^{-1}$. We can choose $V_1, W_1$ small enough so that ${}^t\widetilde{Z}_x(x,t)$ is invertible for all $x\in\mathbb{R}^m$ and $t\in W_1$. From now on we shall write $Z(x,t)$ instead of $\widetilde{Z}(x,t)$, and $V$ and $W$ instead of $V_1$ and $W_1$. So now

\begin{equation*}
    \mathbb{R}\mathrm{T}^\prime_{\mathfrak{W}_t}=\{(z,\zeta)\in\mathbb{C}^m\times\mathfrak{C}_1\;:\;z=Z(x,t),\;\zeta={}^tZ_x(x,t)^{-1}\xi\;\textrm{for some}\;(x,\xi)\in\mathbb{R}^m\times \mathbb{R}^m\},
\end{equation*}

\noindent and we can also assume that the inequality \eqref{eq:bound-derivative-phi} is valid for all $x\in\mathbb{R}^m$. Note that \eqref{eq:well-positioned} is still valid for $(x,t)\in V\times W$. Now we can prove the following inversion formula for the FBI transform:

\begin{thm}
Let $u\in\mathcal{C}^\infty(W;\mathcal{E}^\prime(V))$. Then

\begin{equation}\label{eq:FBI-inversion-formula}
    u(x,t)=\lim_{\epsilon\to 0^+}\frac{1}{(2\pi^3)^\frac{m}{2}}\iint_{\mathbb{R}\mathrm{T}^\prime_{\mathfrak{W}_t}}e^{i\zeta\cdot(Z(x,t)-z^\prime)-\langle\zeta\rangle\langle Z(x,t)-z^\prime\rangle^2-\epsilon\langle\zeta\rangle^2}\mathfrak{F}[u](t;z^\prime,\zeta)\langle\zeta\rangle^\frac{m}{2}\mathrm{d}z^\prime\wedge\mathrm{d}\zeta,
\end{equation}

\noindent where the convergence takes place in $\mathcal{C}^\infty(W;\mathcal{D}^\prime(V))$

\end{thm}

\noindent The proof of this theorem is very close to the one of Lemma IX.$4.1.$ of \cite{trevesbook}.

\begin{proof}
For simplicity let us assume that $u$ is a continuous function. Then for every $\epsilon>0$ we must deal with the integral

\begin{align*}
   & \frac{1}{(2\pi^3)^\frac{m}{2}}\iint_{\mathbb{R}\mathrm{T}^\prime_{\mathfrak{W}_t}}\int_{V}e^{i\zeta\cdot(z^\prime-Z(x^{\prime\prime},t))-\langle\zeta\rangle\langle z^\prime-Z(x^{\prime\prime},t)\rangle^2}e^{i\zeta\cdot(Z(x,t)-z^\prime)-\langle\zeta\rangle\langle Z(x,t)-z^\prime\rangle^2-\epsilon\langle\zeta\rangle^2}\cdot\\
    &\hspace{3cm}\cdot u(x^{\prime\prime},t)\langle\zeta\rangle^\frac{m}{2}\Delta(z^\prime-Z(x^{\prime\prime},t),\zeta)\mathrm{d}Z(x^{\prime\prime},t)\mathrm{d}z^\prime\wedge\mathrm{d}\zeta=\\
    &= \frac{1}{(2\pi^3)^\frac{m}{2}}\iiint_{V\times\left.\mathbb{R}\mathrm{T}^\prime_{\mathfrak{W}_t}\right|_{Z(x^\prime,t)}\times\mathbb{R}^m}e^{i\zeta\cdot(Z(x,t)-Z(x^{\prime\prime},t))-\langle\zeta\rangle\langle Z(x^\prime,t)-Z(x^{\prime\prime},t)\rangle^2-\langle\zeta\rangle\langle Z(x,t)-Z(x^\prime,t)\rangle^2-\epsilon\langle\zeta\rangle^2}\cdot\\
    &\hspace{3cm}\cdot u(x^{\prime\prime},t)\langle\zeta\rangle^\frac{m}{2}\Delta(Z(x^\prime,t)-Z(x^{\prime\prime},t),\zeta)\mathrm{d}Z(x^{\prime\prime},t)\mathrm{d}\zeta\mathrm{d}Z(x^\prime,t)
\end{align*}

\noindent First we change the domain of the integration in the variable $\zeta$ from $\left.\mathbb{R}\mathrm{T}^\prime_{\mathfrak{W}_t}\right|_{Z(x^\prime,t)}$ to $\left.\mathbb{R}\mathrm{T}^\prime_{\mathfrak{W}_t}\right|_{Z(x,t)}$. So we can change the order of integration and integrate in $Z(x^\prime,t)$ first, and we shall calculate

\begin{equation}\label{eq:gaussian-inversion-formula-proof}
   \frac{\langle\zeta\rangle^\frac{m}{2}}{\pi^\frac{m}{2}} \int_{\mathbb{R}^m}e^{-\langle\zeta\rangle [\langle Z(x^\prime,t)-Z(x^{\prime\prime},t)\rangle^2+\langle Z(x,t)-Z(x^\prime,t)\rangle^2]}\Delta(Z(x^\prime,t)-Z(x^{\prime\prime},t),\zeta)\mathrm{d}Z(x^\prime,t).
\end{equation}

\noindent To do so we start noticing that 

\begin{equation*}
     \frac{\omega^\frac{m}{2}}{\pi^\frac{m}{2}} \int_{\mathbb{R}^m}e^{-\omega\langle Z(x^\prime,t)-z\rangle^2}\mathrm{d}Z(x^\prime,t)=1,
\end{equation*}

\noindent for every $\omega\in\mathbb{C}$, with $\Re\; \omega>0$, and $z\in\mathbb{C}^m$, here note that the imaginary part of $Z(x,t)$ has compact support, and also that

\begin{equation*}
     \frac{\omega^\frac{m}{2}}{\pi^\frac{m}{2}} \int_{\mathbb{R}^m}e^{-\omega\langle Z(x^\prime,t)\rangle^2}P(Z(x^\prime,t))\mathrm{d}Z(x^\prime,t)=0,
\end{equation*}

\noindent for every $P(z)$ polynomial such that it has degree one (exactly one) when viewed as a polynomial in each variable separately (in view of Fubini's Theorem). Therefore 

\begin{equation*}
     \frac{\omega^\frac{m}{2}}{\pi^\frac{m}{2}} \int_{\mathbb{R}^m}e^{-\omega\langle Z(x^\prime,t)-z\rangle^2}\Delta(Z(x^\prime,t)-\tilde{z},\zeta)\mathrm{d}Z(x^\prime,t)=\Delta(z-\tilde{z},\zeta),
\end{equation*}

\noindent for every $z,\tilde{z}\in\mathbb{C}^m$. To use this identity we must rewrite $\langle Z(x^\prime,t)-Z(x^{\prime\prime},t)\rangle^2+\langle Z(x,t)-Z(x^\prime,t)\rangle^2$. We start noticing that

\begin{align*}
    \left\langle Z(x^\prime,t)-\frac{(Z(x,t)+Z(x^{\prime\prime},t))}{2}\right\rangle^2&= \left\langle \frac{Z(x^\prime,t)-Z(x,t)}{2}+\frac{Z(x^\prime,t)-Z(x^{\prime\prime},t)}{2}\right\rangle^2\\
    &=\frac{1}{4}\langle Z(x^\prime,t)-Z(x,t)\rangle^2+\frac{1}{4}\langle Z(x^\prime,t)-Z(x^{\prime\prime},t)\rangle^2+\\
    &\quad+\frac{1}{2}(Z(x^\prime,t)-Z(x,t))\cdot(Z(x^\prime,t)-Z(x^{\prime\prime},t))\\
    &=\frac{1}{4}\langle Z(x^\prime,t)-Z(x,t)\rangle^2+\frac{1}{4}\langle Z(x^\prime,t)-Z(x^{\prime\prime},t)\rangle^2-\\
    &\quad-\frac{1}{2}(Z(x,t)-Z(x^\prime,t))\cdot(Z(x^\prime,t)-Z(x^{\prime\prime},t)).
\end{align*}

\noindent So we have obtained the following identity

\begin{align*}
    \langle Z(x^\prime,t)-Z(x,t)\rangle^2+\langle Z(x^\prime,t)-Z(x^{\prime\prime},t)\rangle^2&=4\left\langle Z(x^\prime,t)-\frac{(Z(x,t)+Z(x^{\prime\prime},t))}{2}\right\rangle^2+\\
    &\quad+2(Z(x,t)-Z(x^\prime,t))\cdot(Z(x^\prime,t)-Z(x^{\prime\prime},t)).
\end{align*}

\noindent Also note that 

\begin{align*}
    \langle Z(x,t)-Z(x^{\prime\prime},t)\rangle^2&= \langle Z(x,t)-Z(x^\prime,t)+Z(x^\prime,t)-Z(x^{\prime\prime},t)\rangle^2\\
    &= \langle Z(x^\prime,t)-Z(x,t)\rangle^2+\langle Z(x^\prime,t)-Z(x^{\prime\prime},t)\rangle^2+\\
    &\quad+2(Z(x,t)-Z(x^\prime,t))\cdot(Z(x^\prime,t)-Z(x^{\prime\prime},t)).
\end{align*}

\noindent Summing up these two identities we have that

\begin{align}\label{eq:identity_sum_two_gaussians}
    \langle Z(x^\prime,t)-Z(x,t)\rangle^2+&\langle Z(x^\prime,t)-Z(x^{\prime\prime},t)\rangle^2=\\\nonumber
    &=2\left\langle Z(x^\prime,t)-\frac{(Z(x,t)+Z(x^{\prime\prime},t))}{2}\right\rangle^2+\\\nonumber
    &+\frac{1}{2}\langle Z(x,t)-Z(x^{\prime\prime},t)\rangle^2.
\end{align}

\noindent Now we can calculate \eqref{eq:gaussian-inversion-formula-proof}:

\begin{align*}
     \frac{\langle\zeta\rangle^\frac{m}{2}}{\pi^\frac{m}{2}} \int_{\mathbb{R}^m}&e^{-\langle\zeta\rangle [\langle Z(x^\prime,t)-Z(x^{\prime\prime},t)\rangle^2+\langle Z(x,t)-Z(x^\prime,t)\rangle^2]}\Delta(Z(x^\prime,t)-Z(x^{\prime\prime},t),\zeta)\mathrm{d}Z(x^\prime,t)=\\
     &=e^{-\frac{1}{2}\langle\zeta\rangle\langle Z(x,t)-Z(x^{\prime\prime},t)\rangle^2}\frac{\langle\zeta\rangle^\frac{m}{2}}{\pi^\frac{m}{2}} \int_{\mathbb{R}^m}e^{-2\langle\zeta\rangle\left\langle Z(x^\prime,t)-\frac{(Z(x,t)+Z(x^{\prime\prime},t))}{2}\right\rangle^2 }\cdot\\
     &\quad\cdot\Delta(Z(x^\prime,t)-Z(x^{\prime\prime},t),\zeta)\mathrm{d}Z(x^\prime,t)\\
     &=\frac{e^{-\frac{1}{2}\langle\zeta\rangle\langle Z(x,t)-Z(x^{\prime\prime},t)\rangle^2}}{2^\frac{m}{2}}\Delta\left(\left(\frac{Z(x,t)-Z(x^{\prime\prime},t)}{2}\right),\zeta\right).
\end{align*}

\noindent So let $\epsilon>0$. We have that

\begin{align*}
    \frac{1}{(2\pi^3)^\frac{m}{2}}&\iiint_{V\times\left.\mathbb{R}\mathrm{T}^\prime_{\mathfrak{W}_t}\right|_{Z(x^\prime,t)}\times\mathbb{R}^m}e^{i\zeta\cdot(Z(x,t)-Z(x^{\prime\prime},t))-\langle\zeta\rangle\langle Z(x^\prime,t)-Z(x^{\prime\prime},t)\rangle^2-\langle\zeta\rangle\langle Z(x,t)-Z(x^\prime,t)\rangle^2-\epsilon\langle\zeta\rangle^2}\cdot\\
    &\cdot u(x^{\prime\prime},t)\langle\zeta\rangle^\frac{m}{2}\Delta(Z(x^\prime,t)-Z(x^{\prime\prime},t),\zeta)\mathrm{d}Z(x^{\prime\prime},t)\mathrm{d}\zeta\mathrm{d}Z(x^\prime,t)=\\
    &=\frac{1}{(4\pi^2)^{\frac{m}{2}}}\int_{\left.\mathbb{R}\mathrm{T}^\prime_{\mathfrak{W}_t}\right|_{Z(x^\prime,t)}}\int_Ve^{i\zeta\cdot(Z(x,t)-Z(x^{\prime\prime},t))-\frac{1}{2}\langle\zeta\rangle\langle Z(x,t)-Z(x^{\prime\prime},t)\rangle^2-\epsilon\langle\zeta\rangle^2}u(x^{\prime\prime},t)\cdot\\
    &\cdot\Delta\left(\left(\frac{Z(x,t)-Z(x^{\prime\prime},t)}{2}\right),\zeta\right)\mathrm{d}Z(x^{\prime\prime},t)\mathrm{d}\zeta=\\
    &=\frac{1}{(2\pi)^m}\int_{\left.\mathbb{R}\mathrm{T}^\prime_{\mathfrak{W}_t}\right|_{Z(x^\prime,t)}}e^{-\epsilon\langle\zeta\rangle^2}\mathfrak{F}^{\frac{1}{2}}[u](t;Z(x,t),\zeta)\mathrm{d}\zeta\\
    &\underset{\epsilon\to 0^+}{\longrightarrow}u(x,t).
\end{align*}

\end{proof}

Now we can state the main theorem of this work:

\begin{thm}\label{thm:gevrey-FBI-hypo}
Let $\mathcal{V}$ be a locally integrable structure on $\Omega\subset\mathbb{R}^N$, and let $p\in\Omega$ be an arbitrary point. Consider $(V\times W, x_1,\dots, x_m, t_1,\dots t_n)$ a local coordinate system vanishing at $p$, as described above. Let $u\in\mathcal{C}^\infty(W;\mathcal{D}^\prime(V))$ be a solution of

\begin{equation*}
    \begin{cases}
    \mathrm{L}_1 u = f_1,\\
    \qquad\vdots\\
    \mathrm{L}_n u = f_n,
    \end{cases}
\end{equation*}

\noindent where $f_j\in\mathrm{G}^s(U;\mathrm{L}_1,\cdots,\mathrm{L}_n,\mathrm{M}_1,\dots,\mathrm{M}_m)$, $j=1,\dots,n$. The following are equivalent

\begin{enumerate}
    \item There exist $V_0\subset V$, $W_0\subset W$ open balls containing the origin such that $u|_{V_0\times W_0}\in\mathrm{G}^s(V_0\times W_0;\mathrm{L}_1,\cdots,\mathrm{L}_n,\mathrm{M}_1,\cdots,\mathrm{M}_m)$;\\
    \item There exists $V_1\Subset V$ an open ball centered at the origin such that for every $\chi\in\mathcal{C}_c^\infty(V_1)$, with $0\leq\chi\leq 1$ and $\chi\equiv 1$ in some open neighborhood of the origin, there exist $\widetilde{V}\subset V_1$, $\widetilde{W}\subset W$, open balls centered at the origin, and constants $C,\epsilon>0$ such that
    
    \begin{equation*}
        |\mathfrak{F}[\chi u](t;z,\zeta)|\leq Ce^{-\epsilon|\zeta|^{\frac{1}{s}}},\qquad \forall t\in \widetilde{W}, (z,\zeta)\in\left.\mathbb{R}\mathrm{T}^\prime_{\mathfrak{W}_t}\right|_{\widetilde{V}}\setminus 0,
    \end{equation*}
    
    \noindent where $(z,\zeta)\in\left.\mathbb{R}\mathrm{T}^\prime_{\mathfrak{W}_t}\right|_{\widetilde{V}}$ means that $z=Z(x,t)$, $\zeta={}^\mathrm{t}Z_x(x,t)^{-1}\xi$, $\xi\in\mathbb{R}^m\setminus 0$ and $x\in\widetilde{V}$;\\
    \item For every $\chi\in\mathcal{C}_c^\infty(V)$, with $0\leq\chi\leq 1$ and $\chi\equiv 1$ in some open neighborhood of the origin, there exist $\widetilde{V}\subset V$, $\widetilde{W}\subset W$, open balls centered at the origin, constants $C,\epsilon>0$ such that
    
    \begin{equation}\label{eq:FBI-decay}
        |\mathfrak{F}[\chi u](t;z,\zeta)|\leq Ce^{-\epsilon|\zeta|^{\frac{1}{s}}},\qquad \forall t\in \widetilde{W}, (z,\zeta)\in\left.\mathbb{R}\mathrm{T}^\prime_{\mathfrak{W}_t}\right|_{\widetilde{V}}\setminus 0,
    \end{equation}
    
    \noindent where $(z,\zeta)\in\left.\mathbb{R}\mathrm{T}^\prime_{\mathfrak{W}_t}\right|_{\widetilde{V}}$ means that $z=Z(x,t)$, $\zeta={}^\mathrm{t}Z_x(x,t)^{-1}\xi$, $\xi\in\mathbb{R}^m\setminus 0$ and $x\in\widetilde{V}$
\end{enumerate}
\end{thm}

\noindent Before proving this theorem we shall derive a formula for the derivatives of the Gaussian:

\begin{lem}
Let $\lambda>0$ and $\alpha\in\mathbb{Z}_+^m$. Then

\begin{equation}\label{eq:derivative-gaussian}
    \partial_x^\alpha e^{-\lambda|x|^2}=\sum_{l^1_1+2l^1_2=\alpha_1}\cdots\sum_{l^m_1+2l^m_2=\alpha_m}\frac{\alpha!}{l^1_1!l^1_2!\cdots l^m_1!l^m_2!}(-\lambda)^{l^1_1+l^1_2+\cdots+ l^m_1+l^m_2}(2x_1)^{l^1_1}\cdots(2x_m)^{l^m_1}e^{-\lambda|x|^2}.
\end{equation}
\end{lem}

\begin{proof}
Let $x\in\mathbb{R}^m$ and $j=1,\dots,m$. Consider the function $f:\mathbb{R}\longrightarrow\mathbb{R}$ given by

\begin{equation*}
    f(t)=e^{-\lambda\{x_1^2+\cdots+x_{j-1}^2+t^2+x_{j+1}^2+\cdots+x_m^2\}}.
\end{equation*}

\noindent So $f=g\circ h(t)$, where $g(t)=e^{-\lambda t}$, and $h(t)=x_1^2+\cdots+x_{j-1}^2+t^2+x_{j+1}^2+\cdots+x_m^2$. By Fa\`{a} di Bruno's formula (see for instance \cite{bierstone}) we have that 

\begin{align*}
    \partial_{x_j}^{\alpha_j}e^{-\lambda|x|^2}&=f^{(\alpha_j)}(x_j)\\
    &=\sum_{\{l_1+2l_2+\cdots+\alpha_jl_{\alpha_j}=\alpha_j\}}\frac{\alpha_j!}{l_1!\cdots l_{\alpha_j}!}g^{(l_1+\cdots+ l_{\alpha_j)}}(h(x_j))\prod_{i=1}^{\alpha_j}\left(\frac{h^{(i)}(x_j)}{i!}\right)^{l_i}\\
    &=\sum_{l_1+2l_2=\alpha_j}\frac{\alpha_j!}{l_1!l_2!}(-\lambda)^{l_1+l_2}e^{-\lambda|x|^2}(2x_j)^{l_1}.
\end{align*}

\noindent Since the only term in the sum above that depends on the other variables $x_1,\dots,x_{j-1},x_{j+1},\dots,x_m$ is the Gaussian, we can apply this identity for each variable separately, obtaining \eqref{eq:derivative-gaussian}.

\end{proof}

\begin{proof}[Proof of the Theorem]

$\textit{1}.\Rightarrow\textit{2}.$\,:\\

\noindent By Corolary \ref{cor:caetano-decay} we have that there exist $\mathcal{O}\subset\mathbb{C}^m$ an open neighborhood of $\{Z(x,t)\;:\;x\in V_0,\; t\in W_0\}$ on $\mathbb{C}^m$, and  $F(z,t)\in\mathcal{C}^\infty(\mathcal{O}\times W_0)$ such that 

\begin{equation*}
    \begin{cases}
    F(Z(x,t),t)=u(x,t),\quad\forall (x,t)\in V_0\times W_0;\\
    \left|\partial_{\overline{z}}F(z,t)\right|\leq C^{k+1}k!^{s-1}\mathrm{dist}\,(z,\mathfrak{M}_t)^k,\quad\forall k>0, z\in\mathcal{O}, t\in W_0,
    \end{cases}
\end{equation*}

\noindent where $C$ is a positive constant, as in \eqref{eq:caetano-decay}. Set $V_1=V_0$ and let $\chi\in\mathcal{C}^\infty_c(V_1)$ be such that $0\leq\chi\leq1$ and $\chi\equiv 1$ in $V_2\Subset V_1$, an open ball centered at the origin.  We shall estimate

\begin{equation*}
    \mathfrak{F}[\chi u](t;z,\zeta)=\int_{V_1}e^{i\zeta\cdot(z-Z(x,t))-\langle\zeta\rangle\langle z-Z(x,t)\rangle^2}\chi(x)u(x,t)\Delta(z-Z(x,t),\zeta)\mathrm{d}Z(x,t).
\end{equation*}

\noindent To do so we shall deform the contour of integration. So let $(z,\zeta)\in\left.\mathbb{R}\mathrm{T}^\prime_{\mathfrak{M}_t}\right|_{\widetilde{V}}$ be fixed, where $t\in\widetilde{W}$, and $\widetilde{V}\Subset V$, $\widetilde{W}\Subset W$ are open balls centered at the origin to be chosen latter. Let $\lambda>0$ such that the image of the map

\begin{equation*}
    V_1\times W_0\ni (y,t)\mapsto \Theta_\lambda(y,t)\doteq Z(y,t)-i\lambda \mathbb{E}_{V_2}(y)\frac{\zeta}{\langle\zeta\rangle},
\end{equation*}

\noindent is contained in $\mathcal{O}$, where $\mathbb{E}_{V_2}$ is characteristic function of $V_2$. By Stokes theorem we obtain

\begin{align*}
     \mathfrak{F}[\chi u](t;z,\zeta)&=\int_{V_1\setminus V_2}e^{i\zeta\cdot(z-Z(x,t))-\langle\zeta\rangle\langle z-Z(x,t)\rangle^2}\chi(x)u(x,t)\Delta(z-Z(x,t),\zeta)\mathrm{d}Z(x,t)\\
     &\quad+\int_{V_2}e^{i\zeta\cdot(z-Z(x,t))-\langle\zeta\rangle\langle z-Z(x,t)\rangle^2}F(Z(x,t),t)\Delta(z-Z(x,t),\zeta)\mathrm{d}Z(x,t)\\
     &=\underbrace{\int_{V_1\setminus V_2}e^{i\zeta\cdot(z-Z(x,t))-\langle\zeta\rangle\langle z-Z(x,t)\rangle^2}\chi(x)u(x,t)\Delta(z-Z(x,t),\zeta)\mathrm{d}Z(x,t)}_{(1)}\\
     &\quad+\underbrace{\int_{ V_2}e^{i\zeta\cdot(z-\Theta_\lambda(x,t))-\langle\zeta\rangle\langle z-\Theta_\lambda(x,t)\rangle^2}F(\Theta_\lambda(x,t),t)\Delta(z-\Theta_\lambda(x,t),\zeta)\mathrm{d}Z(x,t)}_{(2)}\\
     &\quad+(-1)^{m-1}2i\int_0^\lambda\int_{V_2}e^{i\zeta\cdot(z-\Theta_\sigma(x,t))-\langle\zeta\rangle\langle z-\Theta_\sigma(x,t)\rangle^2}\overline{\partial_z}F(\Theta_\sigma(x,t),t)\cdot\frac{\zeta}{\langle\zeta\rangle}\mathrm{d}Z(x,t)\cdot\\
     &\quad\underbrace{\cdot\Delta(z-\Theta_\sigma(x,t),\zeta)\mathrm{d}\sigma-\hspace{9cm}}_{(3)}\\
     &-\underbrace{\int_0^\lambda\int_{\partial V_2}e^{i\zeta\cdot(z-\Theta_\sigma(x,t))-\langle\zeta\rangle\langle z-\Theta_\sigma(x,t)\rangle^2}F(\Theta_\sigma(x,t),t)\Delta(z-\Theta_\sigma(x,t),\zeta)\mathrm{d}S_{\mathfrak{M}_t}\mathrm{d}\sigma}_{(4)},
\end{align*}

\noindent where $\mathrm{d}S_{\mathfrak{M}_t}$ is the surface measure in $\{Z(x,t)\,:\,x\in\partial V_2\}$. We shall estimate these four integrals separately. Since the estimate for $(1)$ and $(4)$ are very similar, we will estimate them first. We start writing $\widetilde{V}=B_{r}(0)$, so $z=Z(x_0,t)$ for some $x_0\in B_r(0)$. In view of \eqref{eq:well-positioned} and 

\begin{equation*}
    |z-Z(x,t)|\geq |x_0-x|,
\end{equation*}

\noindent for every $x$,  we have that 

\begin{equation*}
    \Im\{\zeta\cdot (z-Z(x,t))+i\langle\zeta\rangle\langle z-Z(x,t)\rangle^2\}\geq c(r_2-r)^2|\zeta|,
\end{equation*}

\noindent for every $x\in V_1\setminus V_2$, where $V_2=B_{r_2}(0)$, and we are choosing $r<r_2$. Therefore

\begin{equation*}
    \left|\int_{V_1\setminus V_2}e^{i\zeta\cdot(z-Z(x,t))-\langle\zeta\rangle\langle z-Z(x,t)\rangle^2}\chi(x)u(x,t)\Delta(z-Z(x,t),\zeta)\mathrm{d}Z(x,t)\right|\leq Ce^{-c(r_2-r)^2|\zeta|}.
\end{equation*}

\noindent Now the exponent of $(4)$ can be written as

\begin{align*}
    i\zeta\cdot(z-\Theta_\sigma(x,t))-\langle\zeta\rangle\langle z-\Theta_\sigma(x,t)\rangle^2&=i\zeta\cdot(z-Z(x,t))-\sigma\langle\zeta\rangle-\langle\zeta\rangle\langle z-Z(x,t)\rangle^2+\\
    &\quad + \sigma^2\langle\zeta\rangle-2i\sigma(z-Z(x,t))\cdot\zeta.
\end{align*}

\noindent Now recall that in $(4)$ we are integrating in $\sigma$ from $0$ to $\lambda$, so $\sigma<\lambda$, and using \eqref{eq:well-positioned} and \eqref{eq:estimate-re-langle-zeta-rangle} we have that 

\begin{align*}
    \Im\{\zeta\cdot(z-\Theta_\sigma(x,t))+i\langle\zeta\rangle\langle z-\Theta_\sigma(x,t)\rangle^2\}&=\Im\{\zeta\cdot(z-Z(x,t))+i\langle\zeta\rangle\langle z-Z(x,t)\rangle^2\}\\
    &\quad+\sigma\Im\{i\langle\zeta\rangle(1-\sigma)-2(z-Z(x,t))\cdot\zeta\}\\
    &\geq c|z-Z(x,t)|^2|\zeta|+\sigma\Re\langle\zeta\rangle(1-\sigma)-\\&\quad-2\sigma\Im\{\zeta\cdot(z-Z(x,t))\}\\
    &\geq c|z-Z(x,t)|^2|\zeta|+\sigma\sqrt{\frac{1-\kappa^2}{1+\kappa^2}}(1-\sigma)|\zeta|-\\
    &\quad-2\sigma|\zeta||z-Z(x,t)|\\
    &\geq |\zeta||z-Z(x,t)|\left\{c|z-Z(x,t)|-2\lambda\right\}\\
    &\geq |\zeta||z-Z(x,t)|\left[c(r_2-r)-2\lambda\right]\\
    &\geq |\zeta|(r_2-r)\left[c(r_2-r)-2\lambda\right],
\end{align*}

\noindent where we are choosing $\lambda$ satisfying $2\lambda<c(r_2-r)$. Therefore

\begin{equation*}
    \left|\int_0^\lambda\int_{\partial V_1}e^{i\zeta\cdot(z-\Theta_\sigma(x,t))-\langle\zeta\rangle\langle z-\Theta_\sigma(x,t)\rangle^2}F(\Theta_\sigma(x,t),t)\Delta(z-\Theta_\sigma(x,t),\zeta)\mathrm{d}S_{\mathfrak{M}_t}\mathrm{\sigma}\right|\leq Ce^{-\epsilon_1|\zeta|},
\end{equation*}

\noindent where $\epsilon_1=(r_2-r)[c(r_2-r)-2\lambda]$. Before estimating $(2)$ and $(3)$, note that the exponent that appears in each of them is similar to the one that we have just estimated. In $(2)$ we have that $x\in V_2$, \textit{i.e.}, $|x|<r_2$, so the exponential have the following estimate:

\begin{align*}
    \left|e^{i\zeta\cdot(z-\Theta_\lambda(x,t))-\langle\zeta\rangle\langle z-\Theta_\lambda(x,t)\rangle^2}\right|&\leq e^{-\left\{c|z-Z(x,t)|^2|\zeta|+\lambda\sqrt{\frac{1-\kappa^2}{1+\kappa^2}}(1-\lambda)|\zeta|-2\lambda|\zeta||z-Z(x,t)|\right\}}\\
    &\leq e^{-|\zeta|\left\{\lambda\sqrt{\frac{1-\kappa^2}{1+\kappa^2}}(1-\lambda)+|z-Z(x,t)|\left[c|z-Z(x,t)|-2\lambda\right]\right\}}.
\end{align*}

\noindent When $c|z-Z(x,t)|\geq 2\lambda$ we have that

\begin{equation*}
 \left|e^{i\zeta\cdot(z-\Theta_\lambda(x,t))-\langle\zeta\rangle\langle z-\Theta_\lambda(x,t)\rangle^2}\right|\leq e^{-|\zeta|\lambda\sqrt{\frac{1-\kappa^2}{1+\kappa^2}}(1-\lambda)},
\end{equation*}

\noindent and when $c|z-Z(x,t)|\leq 2\lambda$,

\begin{align*}
    \left|e^{i\zeta\cdot(z-\Theta_\lambda(x,t))-\langle\zeta\rangle\langle z-\Theta_\lambda(x,t)\rangle^2}\right|&\leq e^{-|\zeta|\left\{\lambda\sqrt{\frac{1-\kappa^2}{1+\kappa^2}}(1-\lambda)-2\lambda|z-Z(x,t)|\right\}}\\
    &\leq e^{-|\zeta|\left\{\lambda\sqrt{\frac{1-\kappa^2}{1+\kappa^2}}(1-\lambda)-\frac{4\lambda^2}{c}\right\}}\\
    &\leq e^{-|\zeta|\lambda\left\{\sqrt{\frac{1-\kappa^2}{1+\kappa^2}}(1-\lambda)-\frac{4\lambda}{c}\right\}}.
\end{align*}

\noindent Combining these two estimates we conclude that

\begin{equation*}
   \left| \int_{ V_1}e^{i\zeta\cdot(z-\Theta_\lambda(x,t))-\langle\zeta\rangle\langle z-\Theta_\lambda(x,t)\rangle^2}F(\Theta_\lambda(x,t),t)\Delta(z-\Theta_\lambda(x,t),\zeta)\mathrm{d}Z(x,t)\right|\leq Ce^{-\lambda\epsilon_2|\zeta|},
\end{equation*}

\noindent where $\epsilon_2=\sqrt{\frac{1-\kappa^2}{1+\kappa^2}}(1-\lambda)-\frac{4\lambda}{c}>0$, decreasing $\lambda$ if necessary.  To estimate $(3)$ we reason as before, so for each $0<\sigma\leq \lambda$ we have that if $|z-Z(x,t)|\geq 2\sigma/c$ then

\begin{equation*}
 \left|e^{i\zeta\cdot(z-\Theta_\sigma(x,t))-\langle\zeta\rangle\langle z-\Theta_\sigma(x,t)\rangle^2}\right|\leq e^{-|\zeta|\sigma\sqrt{\frac{1-\kappa^2}{1+\kappa^2}}(1-\sigma)},
 \end{equation*}
 
 \noindent and if $|z-Z(x,t)|\leq 2\sigma/c$,
 
\begin{align*}
    \left|e^{i\zeta\cdot(z-\Theta_\sigma(x,t))-\langle\zeta\rangle\langle z-\Theta_\sigma(x,t)\rangle^2}\right|&\leq e^{-|\zeta|\left\{\sigma\sqrt{\frac{1-\kappa^2}{1+\kappa^2}}(1-\sigma)-2\sigma|z-Z(x,t)|\right\}}\\
    &\leq e^{-|\zeta|\left\{\sigma\sqrt{\frac{1-\kappa^2}{1+\kappa^2}}(1-\sigma)-\frac{4\sigma^2}{1-\kappa}\right\}}\\
    &\leq e^{-|\zeta|\sigma\left\{\sqrt{\frac{1-\kappa^2}{1+\kappa^2}}(1-\sigma)-\frac{4\sigma}{1-\kappa}\right\}},
\end{align*}

\noindent and since $\sqrt{\frac{1-\kappa^2}{1+\kappa^2}}(1-\sigma)-\frac{4\sigma}{c}\geq \epsilon_2$, for $\sigma<\lambda$, we have that

\begin{equation*}
    \left|e^{i\zeta\cdot(z-\Theta_\sigma(x,t))-\langle\zeta\rangle\langle z-\Theta_\sigma(x,t)\rangle^2}\right|\leq e^{-\sigma\epsilon_2|\zeta|},
\end{equation*}

\noindent for every $x\in V_1$. So for every $k>0$ we have that

\begin{align*}
    \Big|(-1)^{m-1}2i\int_0^\lambda\int_{V_1}&e^{i\zeta\cdot(z-\Theta_\sigma(x,t))-\langle\zeta\rangle\langle z-\Theta_\sigma(x,t)\rangle^2}\overline{\partial_z}F(\Theta_\sigma(x,t),t)\cdot\frac{\zeta}{\langle\zeta\rangle}\Delta(z-\Theta_\sigma(x,t),\zeta)\mathrm{d}Z(x,t)\mathrm{d}\sigma\Big|\leq\\
    &\leq \int_0^\lambda e^{-\sigma\epsilon_2|\zeta|} \sup_{(x,t)\in V_0\times W_0}\left|\overline{\partial_z}F(\Theta_\sigma(x,t),t)\Delta(z-\Theta_\sigma(x,t),\zeta)\right|\mathrm{d}\sigma\cdot\\
    &\quad\cdot2\left|\frac{|\zeta|}{\langle\zeta\rangle}\right|\left|\int_{V_1}\left|\mathrm{d}Z(x,t)\right|\right|\\
    &\leq C \int_0^\lambda e^{-\sigma\epsilon|\zeta|}C^{k+1}k!^{s-1}\mathrm{dist}\,(\Theta_\sigma(x,t),\mathfrak{M}_t)^k\mathrm{d}\sigma\\
    &\leq C^{k+1}k!^{s-1}\int_0^\infty e^{-\sigma\epsilon_2|\zeta|}\left|\frac{\sigma|\zeta|}{\langle\zeta\rangle}\right|^k\mathrm{d}\sigma\\
    &\leq C^{k+1}k!^{s-1}\int_0^\infty e^{-y}\left(\frac{y}{\epsilon_2|\zeta|}\right)^k\frac{1}{\epsilon_2|\zeta|}\mathrm{d}y\\
    &\leq C^{k+1}\frac{k!^s}{(\epsilon_2|\zeta|)^{k+1}}.
\end{align*}

\noindent Since the constant $C>0$ does not depend on $k$, and the above estimate holds for every $k>0$, we have that

\begin{equation*}
    \left|\int_0^\lambda\int_{V_1}e^{i\zeta\cdot(z-\Theta_\sigma(x,t))-\langle\zeta\rangle\langle z-\Theta_\sigma(x,t)\rangle^2}\overline{\partial_z}F(\Theta_\sigma(x,t),t)\cdot\frac{\zeta}{\langle\zeta\rangle}\Delta(z-\Theta_\sigma(x,t),\zeta)\mathrm{d}Z(x,t)\mathrm{d}\sigma\right|\leq Ce^{-\epsilon_3|\zeta|^{\frac{1}{s}}},
\end{equation*}

\noindent for some constants $C, \epsilon_3>0$. Summing up we have obtained the required estimate \eqref{eq:FBI-decay}, with $\widetilde{W}=W_0$, and $\widetilde{V}=B_r(0)$, where $r>0$ is any positive number less than $r_2$, the radius of $V_2$. 

$\textit{2}.\Rightarrow \textit{3}.$\,:
\\
Let $\chi\in\mathcal{C}^\infty_c(V)$, and $\chi_1\in\mathcal{C}^\infty_c(V_1)$ as in $\textit{2}.$ and $\textit{3}.$. Since $\chi\chi_1\in\mathcal{C}^\infty_c(V_1)$, and $\chi\chi_1\equiv 1$ in some open neighborhood of the origin. So we have that

\begin{equation*}
    |\mathfrak{F}[\chi\chi_1 u](t;z,\zeta)|\leq Ce^{-\epsilon|\zeta|^\frac{1}{s}}, \quad t\in \widetilde{W},\, (z,\zeta)\in\left.\mathbb{R}\mathrm{T}^\prime_{\mathfrak{W}_t}\right|_{\widetilde{V}}.
\end{equation*}

\noindent Now note that $\chi-\chi\chi_1\equiv 0$ in some open neighborhood of the origin $V_2\Subset V_1$. Write $V_2=B_\rho(0)$. So if $x^\prime\in B_{\frac{\rho}{2}}(0)$, $x\in V\setminus V_2$, $t\in W$, and $\zeta\in\left.\mathbb{R}\mathrm{T}^\prime_{\mathfrak{W}_t}\right|_x$, we have that

\begin{align*}
    \Im\{\zeta\cdot(Z(x,t)-Z(x^\prime,t))+i\langle\zeta\rangle\langle Z(x,t)-Z(x^\prime,t)\rangle^2\}&\geq c|Z(x,t)-Z(x^\prime,t)|^2|\zeta|\\
    &\geq c\frac{\rho^2}{4}|\zeta|.
\end{align*}

\noindent Therefore if we set $V_3=B_{\frac{\rho}{2}}(0)\cap \widetilde{V}$ we have that

\begin{equation*}
    |\mathfrak{F}[(\chi-\chi\chi_1)u](t;z,\zeta)|\leq Ce^{-\epsilon^\prime|\zeta|},\quad t\in W,\,(z,\zeta)\in\left.\mathbb{R}\mathrm{T}^\prime_{\mathfrak{W}_t}\right|_{V_3}.
\end{equation*}

\noindent Combining these two decays we obtain

\begin{equation*}
    |\mathfrak{F}[\chi u](t;z,\zeta)|\leq Ce^{-\widetilde{\epsilon}|\zeta|^\frac{1}{s}},\quad t\in \widetilde{W},\,(z,\zeta)\in\left.\mathbb{R}\mathrm{T}^\prime_{\mathfrak{W}_t}\right|_{V_3}.
\end{equation*}

$\textit{3}.\Rightarrow \textit{1}.$\,:
\\
Let $\widetilde{V}\subset V$ and $\widetilde{W}$ be open balls centered at the origin,  $\chi\in\mathcal{C}_c^\infty(V)$ with $0\leq \chi\leq 1$, and $\chi\equiv 1$ in an open ball centered at the origin, and $C,\tilde{\epsilon}>0$ for which the following estimate holds 
 \begin{equation*}
     |\mathfrak{F}[\chi u](t;z,\zeta)|\leq Ce^{-\tilde{\epsilon}|\zeta|^\frac{1}{s}},
 \end{equation*}
 
 \noindent for every $z=Z(x,t)$ and $\zeta={}^tZ_x(x,t)^{-1}\xi$, where $x\in \widetilde{V}$, $t\in \widetilde{W}$ and $\xi\in\mathbb{R}^m\setminus 0$.
 Note that we can choose $\mathrm{supp}\;\chi$ as small as we want, keeping in mind that $\widetilde{V}$ depends on $\chi$. 
 Since we already have that $\mathrm{L}_j u\in \mathrm{G}^s(U;\mathrm{L}_1,\dots,\mathrm{L}_n,\mathrm{M}_1,\dots,\mathrm{M}_m)$, we only have to prove that there exist $V_0\subset V$ and $W_0\subset W$, open balls centered at the origin, such that, writing $U_0=V_0\times W_0$, $u|_{U_0}\in\mathrm{G}^s(U_0;\mathrm{M}_1,\dots,\mathrm{M}_m)$, since the complex vector fields $\{\mathrm{L}_1,\dots,\mathrm{L}_n,\mathrm{M}_1,\dots,\mathrm{M}_m\}$ are pair-wise commuting. We write $V_0=B_r(0)$ and $W_0=B_\delta(0)$. By \eqref{eq:FBI-inversion-formula} we have that 

\begin{equation*}
    \chi(x)u(x,t)=\lim_{\epsilon\to 0^+}\frac{1}{(2\pi^3)^\frac{m}{2}}\iint_{\mathbb{R}\mathrm{T}^\prime_{\mathfrak{W}_t}}e^{i\zeta\cdot(Z(x,t)-z^\prime)-\langle\zeta\rangle\langle Z(x,t)-z^\prime\rangle^2-\epsilon\langle\zeta\rangle^2}\mathfrak{F}[\chi u](t;z^\prime,\zeta)\langle\zeta\rangle^\frac{m}{2}\mathrm{d}z^\prime\wedge\mathrm{d}\zeta.
\end{equation*}

\noindent We shall split this integral in three regions:

\begin{align*}
    &Q^1_t\doteq\{(z^\prime,\zeta)\,:\, z=Z(x^\prime,t),\;\zeta={}^t Z_x(x^\prime,t)^{-1}\xi,\;\text{for some}\;|x^\prime|< \tilde{r}\;\text{and}\;\xi\in\mathbb{R}^m\}\\
    &Q^2_t\doteq\{(z^\prime,\zeta)\,:\,z= Z(x^\prime,t),\;\zeta={}^t Z_x(x^\prime,t)^{-1}\xi, \;\text{for some}\;\tilde{r}\leq|x^\prime|<r_0\;\text{and}\;\xi\in\mathbb{R}^m\}\\
     &Q^3_t\doteq\{(z^\prime,\zeta)\,:\,z= Z(x^\prime,t),\;\zeta={}^t Z_x(x^\prime,t)^{-1}\xi, \;\text{for some}\;r_0\leq|x^\prime|\;\text{and}\;\xi\in\mathbb{R}^m\},
    \end{align*}

\noindent where $\tilde{r}$ and $r_0$ are the radii of $\widetilde{V}$ and $V$. For $\epsilon>0$ and $j=1, 2, 3$, we set

\begin{equation*}
    \mathrm{I}_j^\epsilon(x,t)\doteq\frac{1}{(2\pi^3)^\frac{m}{2}}\iint_{Q^j_t}e^{i\zeta\cdot(Z(x,t)-z^\prime)-\langle\zeta\rangle\langle Z(x,t)-z^\prime\rangle^2-\epsilon\langle\zeta\rangle^2}\mathfrak{F}[\chi u](t;z^\prime,\zeta)\langle\zeta\rangle^\frac{m}{2}\mathrm{d}z^\prime\wedge\mathrm{d}\zeta,
\end{equation*}

\noindent so we can write

\begin{equation*}
      \chi(x)u(x,t)=\lim_{\epsilon\to 0^+}\mathrm{I}_1^\epsilon(x,t)+\mathrm{I}_2^\epsilon(x,t)+\mathrm{I}_3^\epsilon(x,t)
\end{equation*}

\noindent To prove $\textit{1}.$ it is enough to prove the following: there exists a sequence $\{\epsilon_j\}_{j\in\mathbb{Z}_+}$ with $\epsilon_j\to 0$ such that $\mathrm{I}_2^{\epsilon_j}$ and $\mathrm{I}_3^{\epsilon_j}$ converge to analytic vectors for $\mathrm{M}_1,\dots,\mathrm{M}_m$, and that $\mathrm{I}_1^\epsilon$ converges to a Gevrey vector for $\mathrm{M}_1,\dots,\mathrm{M}_m$. To do so we shall prove that there exist $\mathrm{G}_2^\epsilon(z,t)$, $\mathrm{G}_3^\epsilon(z,t)$, $\mathrm{G}_2(z,t)$ and $\mathrm{G}_3(z,t)$, holomorphic functions in some open neighborhood of the origin such that $\mathrm{I}_2^\epsilon(x,t)=\mathrm{G}_2^\epsilon(Z(x,t),t)$, $\mathrm{I}_3^\epsilon(x,t)=\mathrm{G}_3^\epsilon(Z(x,t),t)$, and $\mathrm{G}_2^{\epsilon_j}(z,t)\longrightarrow \mathrm{G}_2(z,t)$ and $\mathrm{G}_3^{\epsilon_j}(z,t)\longrightarrow \mathrm{G}_3(z,t)$ uniformly in $z$, for some sequence $\{\epsilon_j\}_{j\in\mathbb{Z}_+}$ satisfying $\epsilon_j\to 0$, and we shall also prove that there exists a positive constant $C$ such that 

\begin{equation*}
    |\mathrm{M}^\alpha \mathrm{I}_1^\epsilon(x,t)|\leq C^{|\alpha|+1}\alpha!^s,\quad\forall\alpha\in\mathbb{Z}_+^m,
\end{equation*}

\noindent for all $(x,t)\in U_0$ and $\epsilon>0$. \\

$\mathrm{I}_2^\epsilon(x,t)$:\\

\noindent Let $(z^\prime,\zeta)\in Q_t^2$. Since $z^\prime=Z(x^\prime,t)$, with $x^\prime\in V$, we can use \eqref{eq:well-positioned} and \eqref{eq:bound-derivative-phi} to obtain

\begin{align*}
    \Im\{\zeta\cdot(Z(0,t)-Z(x^\prime,t))+i\langle\zeta\rangle\langle Z(0,t)-Z(x^\prime,t)\rangle^2\}&\geq c|\zeta||Z(0,t)-Z(x^\prime,t)|^2\\
    &\geq c|\zeta|(1-\mu^2)|x^\prime|^2\\
    &\geq c(1-\mu^2)\tilde{r}|\zeta|,
\end{align*}

\noindent in other words

\begin{equation*}
    \sup_{(z^\prime,\zeta)\in Q_t^2}\frac{ \Im\{\zeta\cdot(Z(0,t)-z^\prime)+i\langle\zeta\rangle\langle Z(0,t)-z^\prime\rangle^2\}}{|\zeta|}\geq c(1-\mu^2)\tilde{r},
\end{equation*}

\noindent and this is valid for every $t\in W$. So there are $\mathcal{O}_1\subset\mathbb{C}^m$ an open neighborhood of the origin and $W_1\Subset W$ an open neighborhood of the origin, such that

\begin{equation*}
    \sup_{(z^\prime,\zeta)\in Q_t^2}\frac{ \Im\{\zeta\cdot(z-z^\prime)+i\langle\zeta\rangle\langle z-z^\prime\rangle^2\}}{|\zeta|}\geq \frac{ c(1-\mu^2)\tilde{r}}{2},\quad\forall z\in\mathcal{O}_1, t\in W_1.
\end{equation*}

\noindent Now using \eqref{eq:first-bound-FBI} we obtain

\begin{equation}\label{eq:estimate-integrand-I_2}
    \left|e^{i\zeta\cdot(z-z^\prime)-\langle\zeta\rangle\langle z-z^\prime\rangle^2}\mathfrak{F}[\chi u](t;z^\prime,\zeta)\langle\zeta\rangle^\frac{m}{2}\right|\leq C(1+|\zeta|)^{k+\frac{m}{2}}e^{-\frac{c(1-\mu^2)\tilde{r}}{2}|\zeta|},
\end{equation}

\noindent for some $k\geq 0$ and for all $z\in \mathcal{O}_1$, $(z^\prime,\zeta)\in Q_t^2$, and $t\in W_1$. Now set 

\begin{equation*}
    \mathrm{G}_2^\epsilon(z,t)\doteq\frac{1}{(2\pi^3)^\frac{m}{2}}\iint_{Q^2_t}e^{i\zeta\cdot(z-z^\prime)-\langle\zeta\rangle\langle z-z^\prime\rangle^2-\epsilon\langle\zeta\rangle^2}\mathfrak{F}[\chi u](t;z^\prime,\zeta)\langle\zeta\rangle^\frac{m}{2}\mathrm{d}z^\prime\wedge\mathrm{d}\zeta,
\end{equation*}

\noindent and

\begin{equation*}
    \mathrm{G}_2(z,t)\doteq\frac{1}{(2\pi^3)^\frac{m}{2}}\iint_{Q^2_t}e^{i\zeta\cdot(z-z^\prime)-\langle\zeta\rangle\langle z-z^\prime\rangle^2}\mathfrak{F}[\chi u](t;z^\prime,\zeta)\langle\zeta\rangle^\frac{m}{2}\mathrm{d}z^\prime\wedge\mathrm{d}\zeta,
\end{equation*}

\noindent for $\epsilon>0$, $z\in\mathcal{O}_1$, and $t\in W_1$. Let $V_1\Subset V$ and $W_2\Subset W_1$ such that $\{Z(x,t)\;:\;(x,t)\in V_1\times W_2\}\subset \mathcal{O}_1$, so $\mathrm{G}_2^\epsilon(Z(x,t),t)=\mathrm{I}_2^\epsilon(x,t)$ for every $(x,t)\in V_1\times W_2$. Define $\mathrm{I}_2(x,t)\doteq\mathrm{G}_2(Z(x,t),t)$, for $(x,t)\in V_1\times W_2$. In view of \eqref{eq:estimate-integrand-I_2} we have that $\mathrm{G}_2^\epsilon(z,t)$ and $\mathrm{G}_2(z,t)$ are holomorphic with respect to $z$, and $\mathrm{G}_2^\epsilon(z,t)\longrightarrow\mathrm{G}_2(z,t)$ uniformly on $\mathcal{O}_1\times W_1$. \\

$\mathrm{I}_3^\epsilon(x,t)$:\\

\noindent We can deform the domain of integration with respect to the variable $\zeta$, moving the contour of the integration from $\left.\mathbb{R}\mathrm{T}^\prime_{\mathfrak{W}_t}\right|_{Z(x^\prime,t)}$ to $\mathbb{R}^m$, obtaining

\begin{align*}
    \mathrm{I}_3^\epsilon(x,t)&=\frac{1}{(2\pi^3)^\frac{m}{2}}\iint_{Q_t^3}e^{i\zeta\cdot(Z(x,t)-z^\prime)-\langle\zeta\rangle\langle Z(x,t)-z^\prime\rangle^2-\epsilon\langle\zeta\rangle^2}\mathfrak{F}[\chi u](t;z^\prime,\zeta)\langle\zeta\rangle^\frac{m}{2}\mathrm{d}z^\prime\wedge\mathrm{d}\zeta\\
    &=\frac{1}{(2\pi^3)^\frac{m}{2}}\int_{\mathbb{R}^m}\int_{r_0\leq|x^\prime|}e^{i\xi\cdot(Z(x,t)-Z(x^\prime,t))-|\xi|\langle Z(x,t)-Z(x^\prime,t)\rangle^2-\epsilon|\xi|^2}\mathfrak{F}[\chi u](t;Z(x^\prime,t),\xi)|\xi|^\frac{m}{2}\mathrm{d}Z(x^\prime,t)\mathrm{d}\xi\\
    &=\frac{1}{(2\pi^3)^\frac{m}{2}}\int_{\mathbb{R}^m}\int_{r_0\leq|x^\prime|}\Big\langle u(x^{\prime\prime},t), \chi(x^{\prime\prime})e^{i\xi\cdot(Z(x,t)-Z(x^{\prime\prime},t)))-|\xi|\big[\langle Z(x,t)-Z(x^\prime,t)\rangle^2+\langle Z(x^\prime,t)-Z(x^{\prime\prime},t)\rangle^2\big]}\cdot\\
    &\cdot e^{-\epsilon|\xi|^2}|\xi|^\frac{m}{2}\Delta(Z(x^\prime,t)-Z(x^{\prime\prime},t),\xi)\det Z_x(x^{\prime\prime},t)\Big\rangle\mathrm{d}Z(x^\prime,t)\mathrm{d}\xi.\\
\end{align*}

\noindent Now for every $\epsilon>0$ we set

\begin{align}\label{eq:defn-G_3-epsilon}
    \mathrm{G}_3^\epsilon(z,t)\doteq&\frac{1}{(2\pi^3)^\frac{m}{2}}\int_{\mathbb{R}^m}\int_{r_0\leq|x^\prime|}\Big\langle u(x^{\prime\prime},t), \chi(x^{\prime\prime})|\xi|^\frac{m}{2}\Delta(Z(x^\prime,t)-Z(x^{\prime\prime},t),\xi)\cdot\\\nonumber
    &\cdot e^{i\xi\cdot(z-Z(x^{\prime\prime},t)))-|\xi|\big[\langle z-Z(x^\prime,t)\rangle^2+\langle Z(x^\prime,t)-Z(x^{\prime\prime},t)\rangle^2\big]-\epsilon|\xi|^2}\det Z_x(x^{\prime\prime},t)\Big\rangle\mathrm{d}Z(x^\prime,t)\mathrm{d}\xi
\end{align}

\noindent for $z\in\mathbb{C}^m$, and $t\in W_2$. As usual, we begin estimating the exponential, but first for $z=Z(0,t)$:

\begin{align*}
    \left|e^{i\xi\cdot(Z(0,t)- Z(x^{\prime\prime},t)-|\xi|\big[\langle Z(0,t)- Z(x^\prime,t)\rangle^2+\langle Z(x^\prime,t)-Z(x^{\prime\prime},t)\rangle^2\big]}\right|&\leq e^{|\xi||\phi(0,t)-\phi(x^{\prime\prime},t)|-|\xi|\big[|x^{\prime}|^2-|\phi(0,t)-\phi(x^\prime,t)|^2\big]}\cdot\\
    &\cdot e^{-|\xi|\big[|x^\prime-x^{\prime\prime}|^2-|\phi(x^\prime,t)-\phi(x^{\prime\prime},t)|^2\big]}\\
    &\leq e^{|\xi|\mu|x^{\prime\prime}|-|\xi|\left[|x^{\prime}|^2-\mu^2|x^\prime|+|x^\prime-x^{\prime\prime}|^2-\mu^2|x^\prime-x^{\prime\prime}|^2\right]}\\
    &\leq e^{-|\xi|\big[(1-\mu^2)|x^\prime-x^{\prime\prime}|^2+(1-\mu^2)|x^\prime|^2-\mu|x^{\prime\prime}|\big]},
\end{align*}

\noindent where $x^{\prime\prime}\in \textrm{supp}\,\chi$, and $r_0\leq|x^\prime|$. Note that the previous argument (for $\mathrm{I_2^\epsilon}$) does not depend on the "size" of $\mathrm{supp}\,\chi$, therefore we can shrink it as we want to. So we can assume that $|x^{\prime\prime}|$ is small enough so 

\begin{equation*}
     \left|e^{i\xi\cdot(Z(0,t)- Z(x^{\prime\prime},t)-|\xi|\big[\langle Z(0,t)- Z(x^\prime,t)\rangle^2+\langle Z(x^\prime,t)-Z(x^{\prime\prime},t)\rangle^2\big]}\right|\leq e^{-|\xi|(1-\mu^2)|x^\prime|^2}.
\end{equation*}

\noindent Now, for $z\in\mathbb{C}^m$ we have that

\begin{align*}
     \bigg|&e^{i\xi\cdot(z- Z(x^{\prime\prime},t)-|\xi|\big[\langle z- Z(x^\prime,t)\rangle^2+\langle Z(x^\prime,t)-Z(x^{\prime\prime},t)\rangle^2\big]}\bigg|= \\
     &\hspace{2cm}=\left|e^{i\xi\cdot(Z(0,t)- Z(x^{\prime\prime},t)-|\xi|\big[\langle Z(0,t)- Z(x^\prime,t)\rangle^2+\langle Z(x^\prime,t)-Z(x^{\prime\prime},t)\rangle^2\big]}\right|\cdot\\
     &\hspace{2cm}\cdot \left|e^{i\xi\cdot(z-Z(0,t))-|\xi|\big[\langle z-Z(0,t)\rangle^2+2i(z-Z(0,t))\cdot(Z(0,t)-Z(x^\prime,t))\big]}\right|\\
     &\hspace{2cm}\leq  e^{-|\xi|(1-\mu^2)|x^\prime|^2}e^{|\xi||z-Z(0,t)|\big[1+|z-Z(0,t)|+2|Z(0,t)-Z(x^\prime,t)|\big]}\\
     &\hspace{2cm}\leq e^{-|\xi|(1-\mu^2)|x^\prime|^2}e^{|\xi||z-Z(0,t)|\big[1+|z-Z(0,t)|+2(1+\mu)|x^\prime|\big]}.
\end{align*}

\noindent By continuity we can choose $\rho>0$ such that if $|z-Z(0,t)|<\rho$, then 

\[\frac{(1-\mu^2)}{2}|x^\prime|^2-|z-Z(0,t)|\big[1+|z-Z(0,t)|+2(1+\mu)|x^\prime|\big]\geq 0,\quad\forall|x^\prime|\geq r_0.\]

\noindent We can shrink, if necessary, $W_2$, such that $\sup_{t\in W_2}|Z(0,t)|<\rho$. So if we define $\mathcal{O}_2\subset\mathbb{C}^m$ as

\begin{equation*}
    \mathcal{O}_2\doteq\left\{z\in\mathbb{C}^m\;:\sup_{t\in W_2}|z-Z(0,t)|<\rho\right\},
\end{equation*}

\noindent then for every $z\in\mathcal{O}_2$, $t\in W_2$, and $r_0\leq|x^\prime|$, we have that

\begin{equation*}
     \left|e^{i\xi\cdot(z- Z(x^{\prime\prime},t)-|\xi|\big[\langle z- Z(x^\prime,t)\rangle^2+\langle Z(x^\prime,t)-Z(x^{\prime\prime},t)\rangle^2\big]}\right|\leq e^{-|\xi|\frac{(1-\mu^2)}{2}|x^\prime|^2}.
\end{equation*}

\noindent  Since $\mathrm{supp}\,\chi$ and $\overline{W_2}$ are compact sets, there exist $k\in\mathbb{Z_+}$ and $C>0$, such that

\begin{align*}
    \Big|\Big\langle u(x^{\prime\prime},t),\chi(x^{\prime\prime})|\xi|^\frac{m}{2}&\Delta(Z(x^\prime,t)-Z(x^{\prime\prime},t),\xi)\det Z_x(x^{\prime\prime},t)\cdot\\
    &\cdot e^{i\xi\cdot(z-Z(x^{\prime\prime},t)))-|\xi|\big[\langle z-Z(x^\prime,t)\rangle^2+\langle Z(x^\prime,t)-Z(x^{\prime\prime},t)\rangle^2\big]-\epsilon|\xi|^2}\Big\rangle\Big|\leq\\
    &\leq C\sum_{|\alpha|\leq k}\sup_{x^{\prime\prime}\in\mathrm{supp}\,\chi}\Big|\partial_{x^{\prime\prime}}^\alpha\Big\{\chi(x^{\prime\prime})|\xi|^\frac{m}{2}\Delta(Z(x^\prime,t)-Z(x^{\prime\prime},t),\xi)\det Z_x(x^{\prime\prime},t)\cdot\\
    &\cdot e^{i\xi\cdot(z-Z(x^{\prime\prime},t)))-|\xi|\big[\langle z-Z(x^\prime,t)\rangle^2+\langle Z(x^\prime,t)-Z(x^{\prime\prime},t)\rangle^2\big]-\epsilon|\xi|^2}\Big\}\Big|\\
    &\leq C_1|\xi|^{k+\frac{m}{2}}e^{-|\xi|\frac{(1-\mu^2)}{2}|x^\prime|^2},
\end{align*}

\noindent for every $z\in\mathcal{O}_2$, and $t\in W_2$, where the constant $C_1>0$ depends on $\mathrm{supp}\,\chi$, and $k$. Therefore the integrand in \eqref{eq:defn-G_3-epsilon} is dominated by 

\begin{equation}\label{eq:integrand-G_3-dominated}
    C_1|\xi|^{k+\frac{m}{2}}e^{-|\xi|\frac{(1-\mu^2)}{4}|x^\prime|^2}e^{-|\xi|\frac{(1-\mu^2)}{4}r_0^2}.
\end{equation}

\noindent Now since the integral of $e^{-|\xi|\frac{(1-\mu^2)}{4}|x^\prime|^2}$, with respect to $x^\prime$, is bounded by a constant times $|\xi|^{-\frac{m}{2}}$, we have that \eqref{eq:integrand-G_3-dominated} is an integrable function with respect to $(x^\prime,\xi)$ in $\mathbb{R}^{m}\times\mathbb{R}^m$. Therefore by Montel's Theorem, we have that there exists a sequence $\{\epsilon_j\}_{j\in\mathbb{Z}_+}$, with $\epsilon_j\to 0$, such that $\mathrm{G}_3^{\epsilon_j}(z,t)\longrightarrow \mathrm{G}_3(z,t)$ uniformly in $\mathcal{O}_2\times W_2$, and $\mathrm{G}_3(z,t)$ is holomorphic with respect to $z$, and it is given by

\begin{align*}
     \mathrm{G}_3(z,t)\doteq&\frac{1}{(2\pi^3)^\frac{m}{2}}\int_{\mathbb{R}^m}\int_{r_0\leq|x^\prime|}\Big\langle u(x^{\prime\prime},t), \chi(x^{\prime\prime})|\xi|^\frac{m}{2}\Delta(Z(x^\prime,t)-Z(x^{\prime\prime},t),\xi)\cdot\\
    &\cdot e^{i\xi\cdot(z-Z(x^{\prime\prime},t)))-|\xi|\big[\langle z-Z(x^\prime,t)\rangle^2+\langle Z(x^\prime,t)-Z(x^{\prime\prime},t)\rangle^2\big]}\det Z_x(x^{\prime\prime},t)\Big\rangle\mathrm{d}x^\prime\mathrm{d}\xi.
\end{align*}

\noindent So if we take $V_2\subset V_1$ and $W_3\subset W_2$ neighborhoods of the origin, such that 

\begin{equation*}
    \{Z(x,t)\;:\;x\in V_2,t\in W_3\}\subset\mathcal{O}_2,
\end{equation*}

\noindent we have that $\mathrm{I}_3^{\epsilon_j}(x,t)\longrightarrow \mathrm{G}_3(Z(x,t),t)$, for every $(x,t)\in V_2\times W_3$.\\

$\mathrm{I}_1^\epsilon(x,t)$: \\

\noindent  Let $(x,t)\in B_{r}(0)\times B_{\delta}(0)$ and $\alpha\in\mathbb{Z}_+^m$. Then 

\begin{equation*}
      \mathrm{M}^\alpha \mathrm{I}_1^\epsilon(x,t)=  \frac{1}{(2\pi^3)^\frac{m}{2}}\iint_{Q^1_t}\mathrm{M}^\alpha \left\{e^{i\zeta\cdot(Z(x,t)-z^\prime)-\langle\zeta\rangle\langle Z(x,t)-z^\prime\rangle^2}\right\}e^{-\epsilon\langle\zeta\rangle^2}\mathfrak{F}[\chi u](t;z^\prime,\zeta)\langle\zeta\rangle^\frac{m}{2}\mathrm{d}\zeta\mathrm{d}z^\prime.
\end{equation*}

\noindent Since the vectors fields $\{\mathrm{M}_1,\dots,\mathrm{M}_m\}$ are pairwise commuting and $\mathrm{M}_jZ_k(x,t)=\delta_{j,k}$, we can use formula \eqref{eq:derivative-gaussian} to calculate $\mathrm{M}^\alpha \left\{e^{i\zeta\cdot(Z(x,t)-z^\prime)-\langle\zeta\rangle\langle Z(x,t)-z^\prime\rangle^2}\right\}$, obtaining

\begin{align*}
    \mathrm{M}^\alpha \mathrm{I}_1^\epsilon(x,t)&= \frac{1}{(2\pi^3)^\frac{m}{2}}\sum_{\beta\leq\alpha}\binom{\alpha}{\beta}\iint_{Q^1_t}\mathrm{M}^{\alpha-\beta} e^{i\zeta\cdot(Z(x,t)-z^\prime)}\mathrm{M}^\beta e^{-\langle\zeta\rangle\langle Z(x,t)-z^\prime\rangle^2}\cdot\\
    &\cdot e^{-\epsilon\langle\zeta\rangle^2}\mathfrak{F}[\chi u](t;z^\prime,\zeta)\langle\zeta\rangle^\frac{m}{2}\mathrm{d}\zeta\mathrm{d}z^\prime\\
    &=\frac{1}{(2\pi^3)^\frac{m}{2}}\sum_{\beta\leq\alpha}\binom{\alpha}{\beta}\sum_{l^1_1+2l^1_2=\beta_1}\cdots\sum_{l^m_1+2\l^m_2=\beta_m}\frac{\beta!}{l^1_1!l^1_2!\cdots l^m_1!l^m_2!}\cdot\\
    &\cdot \iint_{Q^1_t} e^{i\zeta\cdot(Z(x,t)-z^\prime)-\langle\zeta\rangle\langle Z(x,t)-z^\prime\rangle^2-\epsilon\langle\zeta\rangle^2}\mathfrak{F}[\chi u](t;z^\prime,\zeta)\langle\zeta\rangle^\frac{m}{2}\cdot\\
    &\cdot(-\langle\zeta\rangle)^{l^1_1+l^1_2+\cdots+ l^m_1+l^m_2}(i\zeta)^{\alpha-\beta}(2(Z_1(x,t)-z^\prime_1))^{l^1_1}\cdots(2(Z_m(x,t)-z^\prime_m))^{l^m_1}\mathrm{d}\zeta\mathrm{d}z^\prime.\\
\end{align*}

\noindent Therefore by \eqref{eq:FBI-decay} there exists $\tilde{\epsilon}>0$ such that

\begin{align*}
   \left| \mathrm{M}^\alpha \mathrm{I}_1^\epsilon(x,t)\right|&\leq \frac{1}{(2\pi^3)^\frac{m}{2}}\sum_{\beta\leq\alpha}\binom{\alpha}{\beta}\sum_{l^1_1+2l^1_2=\beta_1}\cdots\sum_{l^m_1+2l^m_2=\beta_m}\frac{\beta!}{l^1_1!l^1_2!\cdots l^m_1!l^m_2!}\iint_{Q^1_t}e^{-(1-\kappa)|\zeta||Z(x,t)-z^\prime|^2}\cdot\\
   &\cdot |\zeta|^{|\alpha-\beta|+l^1_1+l^1_2+\cdots+ l^m_1+l^m_2+\frac{m}{2}}\left|\mathfrak{F}[\chi u](t;z^\prime,\zeta)\right||\mathrm{d}\zeta\mathrm{d}z^\prime|\\
   &\leq C_1^{|\alpha|+1}\sum_{\beta\leq\alpha}\binom{\alpha}{\beta}\sum_{l^1_1+2l^1_2=\beta_1}\cdots\sum_{l^m_1+2l^m_2=\beta_m}\frac{\beta!}{l^1_1!l^1_2!\cdots l^m_1!l^m_2!}\iint_{Q^1_t}e^{-\tilde{\epsilon}|\zeta|^{\frac{1}{s}}}\cdot\\
   &\cdot|\zeta|^{|\alpha-\beta|+l^1_1+l^1_2+\cdots+ l^m_1+l^m_2+\frac{m}{2}}|\mathrm{d}\zeta\mathrm{d}z^\prime|\\
   &\leq  C_2^{|\alpha|+1}\sum_{\beta\leq\alpha}\binom{\alpha}{\beta}\sum_{l^1_1+2l^1_2=\beta_1}\cdots\sum_{l^m_1+2l^m_2=\beta_m}\frac{\beta!}{l^1_1!l^1_2!\cdots l^m_1!l^m_2!}\int_{0}^\infty e^{-\tilde{\epsilon}\rho^{\frac{1}{s}}}\cdot\\
   &\cdot\rho^{|\alpha-\beta|+l^1_1+l^1_2+\cdots+ l^m_1+l^m_2+\frac{m}{2}+m-1}\mathrm{d}\rho\\
   &\leq  C_3^{|\alpha|+1}\sum_{\beta\leq\alpha}\binom{\alpha}{\beta}\sum_{l^1_1+2l^1_2=\beta_1}\cdots\sum_{l^m_1+2l^m_2=\beta_m}\frac{\beta!}{l^1_1!l^1_2!\cdots l^m_1!l^m_2!}\frac{\alpha!^s}{\beta!^s}(l^1_1+l^1_2)!^s\cdots(l^m_1+l^m_2)!^s\\
   &=  C_3^{|\alpha|+1}\sum_{\beta\leq\alpha}\binom{\alpha}{\beta}\sum_{l^1_1+2l^1_2=\beta_1}\cdots\sum_{l^m_1+2l^m_2=\beta_m}\frac{\beta!}{l^1_1!(2l^1_2)!\cdots l^m_1!(2l^m_2)!}\frac{\alpha!^s}{\beta!^s}(l^1_1+l^1_2)!^s\frac{(2l^1_2)!}{l^1_2!}\cdots\\
   &\cdots (l^m_1+l^m_2)!^s\frac{(2l^m_2)!}{l^m_2!}\\
   &\leq  C_4^{|\alpha|+1}\sum_{\beta\leq\alpha}\binom{\alpha}{\beta}\sum_{l^1_1+2l^1_2=\beta_1}\cdots\sum_{l^m_1+2\l^m_2=\beta_m}\frac{\beta!}{l^1_1!(2l^1_2)!\cdots l^m_1!(2l^m_2)!}\frac{\alpha!^s}{\beta!^s}(l^1_1+l^1_2)!^sl^1_2!\cdots\\
   &\cdots (l^m_1+l^m_2)!^sl^m_2!\\
   &\leq  C_5^{|\alpha|+1}\sum_{\beta\leq\alpha}\binom{\alpha}{\beta}\sum_{l^1_1+2l^1_2=\beta_1}\cdots\sum_{l^m_1+2l^m_2=\beta_m}\frac{\beta!}{l^1_1!(2l^1_2)!\cdots l^m_1!(2l^m_2)!}\frac{\alpha!^s}{\beta!^s}(l^1_1+2l^1_2)!^s\cdots\\
   &\cdots (l^m_1+2l^m_2)!^s\\
   &=  C_5^{|\alpha|+1}\alpha!^s\sum_{\beta\leq\alpha}\binom{\alpha}{\beta}\sum_{l^1_1+2l^1_2=\beta_1}\cdots\sum_{l^m_1+2l^m_2=\beta_m}\frac{\beta!}{l^1_1!(2l^1_2)!\cdots l^m_1!(2l^m_2)!}\\
   &\leq  C_5^{|\alpha|+1}\alpha!^s\sum_{\beta\leq\alpha}\binom{\alpha}{\beta}\sum_{l^1_1+l^1_2=\beta_1}\cdots\sum_{l^m_1+l^m_2=\beta_m}\frac{\beta!}{l^1_1!l^1_2!\cdots l^m_1!l^m_2!}\\
   &\leq  C_6^{|\alpha|+1}\alpha!^s,
   \end{align*}
   
   \noindent where the constant $C_6$ does not depend on $\epsilon$. The constant $C_6$ can be taken as $3C_5$, in view of Lemma $4.2.$ of \cite{bierstone}.

\end{proof}

\begin{rmk}\label{rmk:FBI-characterization-W}
Note that for the implications $\textit{1}.\Rightarrow\textit{2}.\Rightarrow\textit{3}.$ we can take $\widetilde{W}=W_0$. Also by a closer inspection on the proof of $\textit{1}.\Rightarrow\textit{2}.$ we can take $V_1$ as $V$, so that if $\chi\in\mathcal{C}^\infty_c(V)$ such that $\chi\equiv 1$ on $V_2$ an open ball centered at the origin, then the inequality \eqref{eq:FBI-decay} is valid for every open ball $\widetilde{V}\Subset V_2$, centered at the origin.
\end{rmk}

\section{Propagation of singularities}

In 1983 N. Hanges and F. Treves (\cite{hanges92}) proved that hypo-analytic regularity propagates along elliptic submanifolds, and in their proof they actually showed that the decay of the FBI transform  propagates. But since then all the propagation of singularities results, concerning systems of complex vector fields, were obtained in the setting of CR geometry, for instance holomorphic extendabillity of CR functions, propagation along CR orbits, sector extendability, (see \cite{tumanov1997propagation}, \cite{trepreau1990propagation} and \cite{baracco2005propagation}) and so on. We did not find in the literature any other result concerning propagation of Gevrey singularities in this set up.

We shall consider only analytic tube structures, \textit{i.e}., locally the hypo-analytic structure is given by $Z(x,t)=x+i\phi(t)$, defined on $U=V\times W$, and $\phi(t)$ is analytic. One of the reasons we are only dealing with tube structures is that the real structure bundle, $\mathbb{R}\mathrm{T}^\prime_{\mathfrak{W}_t}$, is trivial for every $t$, \textit{i.e.}, it is equal to $Z(U)\times\mathbb{R}^m$. Now we will recall a simple comparison result for the FBI transform for solutions (see proposition IX.$5.3$., pg $436$ of \cite{trevesbook}):

\begin{prop}\label{prop:propagation-FBI-treves-easy}
There are open balls $V_0\Subset V_1\Subset V$ in $\mathbb{R}^m$ and $W_0\Subset W$ in $\mathbb{R}^n$, all centered at the origin, and constants $r,\kappa,R>0$ such that, if $\chi\in\mathcal{C}_c^\infty(V_1)$ is equal to $1$ in $V_0$, then, to every solution $u$ in $U=V\times W$, there is a constant $C>0$ such that

\begin{equation*}
    |\mathfrak{F}[\chi u](t;z,\zeta)-\mathfrak{F}[\chi u](t^\prime;z,\zeta)|\leq Ce^{-|\zeta|/R},
\end{equation*}

\noindent in the region

\begin{equation*}
    t, t^\prime\in W_0, z\in\mathbb{C}^m, |z|<r, \zeta\in\mathfrak{C}_\kappa.
\end{equation*}
\end{prop}

\noindent This proposition can be used to show that hypo-analyticity propagates along connected fibers (recall that a fiber is locally a level set of the map $Z(x,t)$). Let us just indicate how it is done. Suppose that $u$ is hypo-analytic at the origin and let $t_0\in W_0$ be such that $Z(0,t_0)=0$. To show that $u$ is hypo-analytic at $(0,t_0)$ it is enough to show that $u|_{\mathcal{H}_{t_0}}$ is hypo-analytic at $(0,t_0)$, where $\mathcal{H}_{t_0}=\{(x,t_0)\,:\,x\in V_0\}$, but this is equivalent to 

\begin{equation*}
    |\mathfrak{F}[\chi u](t_0;z,\zeta)|\leq Ce^{-\epsilon|\zeta|},
\end{equation*}

\noindent for some $C,\epsilon>0$, and $z$ in some open neighborhood of the origin and $\zeta\in\mathfrak{C}_\kappa$, for some $0<\kappa<1$. But since $u$ is hypo-analytic at the origin, we have that 

\begin{equation*}
     |\mathfrak{F}[\chi u](0;z,\zeta)|\leq Ce^{-\epsilon|\zeta|},
\end{equation*}

\noindent for some $C,\epsilon>0$, and $z$ in some open neighborhood of the origin and $\zeta\in\mathfrak{C}_\kappa$, for some $0<\kappa<1$, therefore we have the desired decay at $t_0$ in view of Proposition \ref{prop:propagation-FBI-treves-easy}. One can follow the end of the proof of the Theorem \ref{thm:propagation-Gevrey} to globalize this argument to connected fibers. So why we can not use this same argument for Gevrey vectors? First, we do not have the property that ensures the desired regularity by only looking to restrictions on maximally real submanifolds. And the second reason is that for Gevrey regularity, to use a FBI transform argument, $(z,\zeta)$ must belong to the real structure bundle $\mathbb{R}\mathrm{T}^\prime_{\mathfrak{W}_t}$, that depends on $t$. So to avoid this dependence we are restringing ourselves to tube structures. To deal with "restringing to maximally real submanifolds is not enough" problem we need some sort of foliation near the "propagators", and for that it is important for the structure to be analytic.

\section{Propagation of Gevrey regularity for solutions of the non homogeneous system}

In this section we will define the sets that will propagate the Gevrey regularity, the "propagators", and then exhibit the proof of the second main theorem of this work. Let $\Sigma\subset\Omega$ be a connected subset of $\Omega$, satisfying the following properties:

\begin{enumerate}
    \item For every $p\in\Sigma$ there is $(U,Z)$, a hypo-analytic chart, with $p\in U$, such that $\Sigma\cap U\subset Z^{-1}(0)$;\\
    \item In the same situation as above, for every $q\in\Sigma\cap U$, and $\widetilde{U}_1\Subset U$, an open neighborhood of $p$, there is $\widetilde{U}_2\Subset U$, an open neighborhood of $q$, such that the connected component of the fiber $Z^{-1}(Z(q^\prime))$ that contains $q^\prime$ intersects $\widetilde{U}_1$, for every $q^\prime\in \widetilde{U}_2$;\\
    \item the map $\Sigma\ni p\mapsto \sup\{r>0\;:\;B_r(p)\subset U\}$ is continuous.
\end{enumerate}

 \noindent Condition $(3)$ is not exactly a condition, because we can always shrink the open set $U$ for each $p$, therefore we can choose $U$ to be a ball with radius varying continuously on $p$. Condition $(2)$ implies that for every $q^\prime\in \widetilde{U}_2$ there is a curve $\gamma_{q^\prime}:[0,1]\longrightarrow U$ satisfying
    \begin{itemize}
        \item $\gamma_{q^\prime}(0)=q^\prime$;\\
        \item $Z(\gamma_{q^\prime}(\sigma))=Z(q^\prime)$, for every $0\leq \sigma\leq 1$;\\
        \item $\gamma_{q^\prime}(1)\in \widetilde{U}_1$;
    \end{itemize}
    
\noindent Since the structure is analytic, the level sets of $Z(x,t)$ are subanalytic sets, therefore the curves $\{\gamma_{q^\prime}\}_{q^\prime\in\widetilde{U}_2}$ have bounded length, see for instance section $8$ of \cite{hardt1982some} or pg. $39$ of \cite{treves1983local} (in the appendix wrote by B. Teissier). Let $p\in\Sigma$, and let $(U,Z)$ be the hypo-analytic chart described above. Take local coordinates in $(U$, $x_1, \dots, x_m, t_1, \dots, t_n)$, such that in this coordinates $p=0$, $U=V\times W$, and the real structure bundle on $V$, $\left.\mathbb{R}\mathrm{T}^\prime_{\mathfrak{W}_t}\right|_V$, is well positioned for every $t\in W$, \textit{i.e.}, there exists $c_0>0$ such that

\begin{equation*}
    \Im\{\xi\cdot(Z(x,t)-Z(y,t))+i|\xi|\langle Z(x,t)-Z(y,t)\rangle^2\}\geq c_0|\xi||Z(x,t)-Z(y,t)|^2,
\end{equation*}

\noindent for every $x,y\in V$, $t\in W$, and $\xi\in\mathbb{R}^m$. 

\begin{lem}
Let $\rho>0$ be such that $B_{\rho}(0)\Subset V$, and let $f\in\mathcal{C}^\infty(W;\mathcal{E}^\prime(K\setminus B_\rho(0)))$, where $B_{\rho}(0)\subset K\Subset V$, is a compact set. Then

\begin{equation*}
    |\mathfrak{F}[f](t;Z(x,t),\xi)|\leq C e^{-\epsilon|\xi|},\quad \forall x\in B_{\rho/2}(0),\, t\in W,\,\xi\in\mathbb{R}^m.
\end{equation*}
\end{lem}

\begin{proof}
Let $x\in B_{\rho/2}(0)$ and $y\in K\setminus B_{\rho}(0)$, then

\begin{align*}
    \Im\{\xi\cdot(Z(x,t)-Z(y,t))+i|\xi|\langle Z(x,t)-Z(y,t)\rangle^2\}&\geq c_0|\xi||Z(x,t)-Z(y,t)|^2\\
    &\geq c_0|\xi||x-y|^2\\
    &\geq c_0|\xi|\frac{\rho^2}{4},
\end{align*}

\noindent for every $t\in W$ and $\xi\in\mathbb{R}^m$. Therefore

\begin{align*}
    |\mathfrak{F}[f](t;Z(x,t),\xi)|&=\Big|\Big\langle f(y,t), e^{i\xi\cdot(Z(x,t)-Z(y,t))-|\xi|\langle Z(x,t)-Z(y,t)\rangle^2}\Delta(Z(x,t)-Z(y,t),\xi)\cdot\\
    &\cdot\det Z_x(y,t)\Big\rangle\Big|\\
    &\leq C\sum_{|\alpha|\leq\lambda}\sup_{y\in K\setminus B_{\rho/2}(0)}\Big|\partial_y^\alpha\Big\{e^{i\xi\cdot(Z(x,t)-Z(y,t))-|\xi|\langle Z(x,t)-Z(y,t)\rangle^2}\cdot\\
    &\cdot \Delta(Z(x,t)-Z(y,t),\xi)\det Z_x(y,t)\Big\}\Big|\\
    &\leq C|\xi|^\lambda e^{-c_0|\xi|\frac{\rho^2}{4}}\\
    &\leq C e^{-\frac{c_0\rho^2}{8}|\xi|},
\end{align*}

\noindent for every $x\in B_{\rho/2}(0)$, $t\in W$, and $\xi\in\mathbb{R}^m$.

\end{proof}

\begin{thm}\label{thm:propagation-Gevrey}
Let $\Omega\subset\mathbb{R}^{n+m}$ be an open set endowed with an analytic hypo-analytic structure of tube type. Let $\Sigma\subset\Omega$ be a connected submanifold as described above. If $u\in\mathcal{D}^\prime(\Omega)$ is such that $\mathbb{L}u\in \mathrm{G}^s(\Omega)$, then $\textrm{singsupp}_s\,u\cap \Sigma=\emptyset$ or $\Sigma\subset\textrm{singsupp}_s\,u$. 
\end{thm}

\begin{proof}
Let $p\in\Sigma$, and suppose that $p\notin\textrm{singsupp}_s\,u$. Let $(U,Z)$ be the hypo-analytic chart described before. Consider in $U$ the local coordinates $(x_1, \dots, x_m, t_1, \dots, t_n)$, and the complex vector fields $\{\mathrm{M}_1, \dots, \mathrm{M}_m, \mathrm{L}_1, \dots, \mathrm{L}_n\}$, as in the previous chapter. In this coordinates system $p=0$, and we write $U=V\times W$, where $V\subset \mathbb{R}^n$ and $W\subset\mathbb{R}^m$ are both open neighborhoods of the origin. We also have that

\begin{equation*}
    Z_k(x,t)=x_k+i\phi_k(t),\quad k=1,\dots, m,
\end{equation*}

\noindent and 

\begin{center}
\begin{tabular}{ll}
    $\mathrm{L}_j Z_k=0$  & $\mathrm{M}_l Z_k = \delta_{l,k}$\\
   $\mathrm{L}_j t_i = \delta_{j,i}$  & $\mathrm{M}_l t_i=0$.
\end{tabular}
\end{center}

\noindent We can also assume that the real structure bundle $\mathbb{R}\mathrm{T}^\prime_{\mathfrak{W}_t}$ is well positioned, for every $t\in W$. Now let $\rho>0$ be such that $B_\rho(0)\subset V$, and $\chi\in\mathcal{C}^\infty_c(V)$ be such that $\chi\equiv 1$ on $B_\rho(0)$.
Since $\mathbb{L}u\in \mathrm{G}^s(\Omega)$, we have that $\mathrm{L}_j u\in\mathrm{G}^s(U; \mathrm{L}_1, \dots, \mathrm{L}_n, \mathrm{M}_1, \dots, \mathrm{M}_m)$, then $u\in\mathcal{C}^\infty(W; \mathcal{D}^\prime(V))$. By Theorem \ref{thm:gevrey-FBI-hypo} and Remark \ref{rmk:FBI-characterization-W} we have that

\begin{equation}\label{eq:FBI-chiLu}
    \big|\mathfrak{F}[\chi\mathrm{L}_ju](t;Z(x,t),\xi)\big|\leq Ce^{-\epsilon_1|\xi|^{\frac{1}{s}}},\quad\forall x\in B_{\rho^\prime}(0),\forall t\in W,\forall\xi\in\mathbb{R}^m, j=1,\dots,n,
\end{equation}

\noindent for some $C,\epsilon_1>0$, where $\rho/2\leq\rho^\prime<\rho$. We are assuming that $u|_{U_0}\in \mathrm{G}^s(U_0; \mathrm{L}_1, \dots, \mathrm{L}_n, \mathrm{M}_1, \dots, \mathrm{M}_m)$, for some open neighborhood of the origin, $U_0=V_0\times W_0$. Then by Theorem \ref{thm:gevrey-FBI-hypo} and Remark \ref{rmk:FBI-characterization-W} there exist $V_1\Subset V$, an open neighborhood of the origin, and positive constants $C,\epsilon_2$, such that

\begin{equation}\label{eq:FBI-u-decay-propagation-hypothesis}
    |\mathfrak{F}[\chi u](t;Z(x,t),\xi)|\leq Ce^{-\epsilon_2|\xi|^\frac{1}{s}},\quad\forall (x,t)\in V_1\times W_0,\,\forall \xi\in\mathbb{R}^m.
\end{equation}

\noindent By condition $(2)$, for every $(x_0,t_0)\in\Sigma\cap\left( B_{\rho/2}(0)\times W\right)$ there exists $\widetilde{V}\times \widetilde{W}\subset B_{\rho/2}\times W$, an open neighborhood of $(x_0,t_0)$, such that, for every $(x^\prime,t^\prime)\in \widetilde{V}\times \widetilde{W}$ there is a curve $\gamma_{(x^\prime,t^\prime)}:[0,1]\longrightarrow U$, satisfying:

\begin{itemize}
    \item $\gamma_{(x^\prime,t^\prime)}(0)=(x^\prime,t^\prime)$;\\
    \item $Z(\gamma_{(x^\prime,t^\prime)}(\sigma))=Z(x^\prime,t^\prime)$, for every $0\leq\sigma\leq 1$;\\
    \item $\gamma_{(x^\prime,t^\prime)}(1)\in V_1\times W_0$;\\
    \item There exists $C_1>0$ such that 

\begin{equation*}
    \int_0^1\|\gamma^\prime_{(x^\prime,t^\prime)}(\sigma)\|\mathrm{d}\sigma\leq C_1,
\end{equation*}

\noindent for every $(x^\prime,t^\prime)\in \widetilde{V}\times \widetilde{W}$.
\end{itemize}

\noindent Now let $(x^\prime,t^\prime)\in\widetilde{V}\times\widetilde{W}$ be fixed. We write $\gamma_{(x^\prime,t^\prime)}(\sigma)=(\gamma^{(1)}_{(x^\prime,t^\prime)}(\sigma),\gamma^{(2)}_{(x^\prime,t^\prime)}(\sigma))$. By Stokes theorem we have that

\begin{align*}
    \mathfrak{F}[\chi u](t^\prime;Z(x^\prime,t^\prime),\xi)-&\mathfrak{F}[\chi u](\gamma^{(2)}_{(x^\prime,t^\prime)}(1), Z(\gamma_{(x^\prime,t^\prime)}(1)),\xi)=\\
    &=\int_0^1\frac{\partial}{\partial \sigma}\mathfrak{F}[\chi u](\gamma^{(2)}_{(x^\prime,t^\prime)}(\sigma);Z(x^\prime,t^\prime),\xi)\mathrm{d}\sigma\\
    &=\int_0^1\sum_{j=1}^n\mathfrak{F}[\mathrm{L}_j(\chi u)](\gamma^{(2)}_{(x^\prime,t^\prime)}(\sigma);Z(x^\prime,t^\prime),\xi)\frac{\mathrm{d}}{\mathrm{d}\sigma}{\gamma^{(2)}_{(x^\prime,t^\prime)}}_j(\sigma)\mathrm{d}\sigma\\
    &=\int_0^1\sum_{j=1}^n\mathfrak{F}[u\mathrm{L}_j\chi ](\gamma^{(2)}_{(x^\prime,t^\prime)}(\sigma);Z(x^\prime,t^\prime),\xi)\frac{\mathrm{d}}{\mathrm{d}\sigma}{\gamma^{(2)}_{(x^\prime,t^\prime)}}_j(\sigma)\mathrm{d}\sigma+\\
    &+\int_0^1\sum_{j=1}^n\mathfrak{F}[\chi\mathrm{L}_ju](\gamma^{(2)}_{(x^\prime,t^\prime)}(\sigma);Z(x^\prime,t^\prime),\xi)\frac{\mathrm{d}}{\mathrm{d}\sigma}{\gamma^{(2)}_{(x^\prime,t^\prime)}}_j(\sigma)\mathrm{d}\sigma.
\end{align*}

\noindent Now we analyze these two terms separately. First we note that $\mathrm{L}_j\chi$ vanishes on $B_\rho(0)$, for $j=1,\dots, n$ therefore, by the previous lemma, we have that there exist $C,\epsilon_3>0$ such that

\begin{equation*}
     |\mathfrak{F}[u\mathrm{L}_j\chi ](t;Z(x,t),\xi)|\leq Ce^{-\epsilon_3|\xi|},\quad \forall x\in B_{\rho/2}(0),\,t\in W,\, \xi\in\mathbb{R}^m,\,j=1,\dots, n.
\end{equation*}

\noindent Therefore 

\begin{equation*}
     |\mathfrak{F}[u\mathrm{L}_j\chi ](\gamma^{(2)}_{(x^\prime,t^\prime)}(\sigma);Z(x^\prime,t^\prime),\xi)|\leq Ce^{-\epsilon_3|\xi|},\quad  0\leq\sigma\leq 1,\, \forall\xi\in\mathbb{R}^m,\,j=1,\dots, n.
\end{equation*}

\noindent In view of \eqref{eq:FBI-chiLu} we also have that

\begin{equation*}
     |\mathfrak{F}[\chi\mathrm{L}_ju](\gamma^{(2)}_{(x^\prime,t^\prime)}(\sigma);Z(x^\prime,t^\prime),\xi)|\leq Ce^{-\epsilon_1|\xi|^\frac{1}{s}},\quad  0\leq\sigma\leq 1,\, \forall\xi\in\mathbb{R}^m,\,j=1,\dots, n.
\end{equation*}

\noindent Summing up we have obtained

\begin{align*}
    |\mathfrak{F}[\chi u](t^\prime;Z(x^\prime,t^\prime),\xi)|&\leq |\mathfrak{F}[\chi u](\gamma^{(2)}_{(x^\prime,t^\prime)}(1), Z(\gamma_{(x^\prime,t^\prime)}(1)),\xi)|+ Ce^{-\epsilon_4|\xi|^\frac{1}{s}}\\
    &\leq Ce^{-\epsilon|\xi|^\frac{1}{s}}
\end{align*}

\noindent for every $(x^\prime,t^\prime)\in \widetilde{V}\times\widetilde{W}$, since $\gamma_{(x^\prime,t^\prime)}(1)\in V_1\times W_0$, where $\epsilon_4=\min\{\epsilon_1,\epsilon_3\}$ and $\epsilon=\min\{\epsilon_2,\epsilon_4\}$. So we conclude that for $p\in\Sigma$ there exists a neighborhood $\mathcal{U}\Subset U$, such that if $p\notin \textrm{singsupp}_s\,u$, then $\Sigma\cap \mathcal{U}\subset\complement\,\textrm{singsupp}_s\,u$. Moreover, since the map $\Sigma\ni p\mapsto\sup\{r>0\;:\;B_r(p)\subset U\}$ is continuous, the same can be assumed for the map $\Sigma\ni p\mapsto \sup\{r>0\;:\;B_r(p)\subset \mathcal{U}\}$. To indicate the dependence of $p$ in $\mathcal{U}$, we shall write $\mathcal{U}=\mathcal{U}(p)$. Now we claim that $\Sigma\cap\,\textrm{singsupp}_s\,u$ is an open set. So take $\{p_k\}_{k\in\mathbb{Z}_+}$ a sequence on $\Sigma\cap\complement\,\textrm{singsupp}_s\,u$, such that $p_k\rightarrow p\in\Sigma$.
Now there exists $\delta>0$ such that the open set $\mathcal{U}(p_k)$, as described above, contains a ball, centered at $p_k$, of radius at least $\delta$, for every $k$. So there exists $k>0$ such that $p\in {\mathcal{U}}(p_k)$. Since $p_k\notin \textrm{singsupp}_s\,u$ we have that $\Sigma\cap \mathcal{U}(p_k)\subset\complement\,\textrm{singsupp}_s\,u$, \textit{i.e}, $p\notin\textrm{singsupp}_s u$. Clearly $\Sigma\cap\,\textrm{singsupp}_s u$ is closed. Therefore $\Sigma\cap\,\textrm{singsupp}_s u=\emptyset$ or $\Sigma\subset\textrm{singsupp}_s u$.

\end{proof}

\section{Examples}

Consider in $\mathbb{R}^2$ a real valued, real-analytic function $\phi$ satisfying $\phi(0)=0$, and consider in $\mathrm{R}^3$ the structure $\mathcal{V}$ defined by the complex vector fields

\begin{align*}
    &\mathrm{L}_1= \frac{\partial}{\partial t_1}-i2\phi(t_1,t_2)\frac{\partial \phi}{\partial t_1}(t_1,t_2)\frac{\partial}{\partial x},\\
    &\mathrm{L}_2= \frac{\partial}{\partial t_2}-i2\phi(t_1,t_2)\frac{\partial \phi}{\partial t_2}(t_1,t_2)\frac{\partial}{\partial x}.
\end{align*}

\noindent The first integral for these complex vector fields is given by

\begin{equation*}
    Z(x,t_1,t_2)=x+i\phi(t_1,t_2)^2,
\end{equation*}

\noindent and the characteristic set $\mathrm{T}^0$ is equal to 

\begin{equation*}
   \mathrm{T}^0= \{(x,t_1,t_2,\xi,\eta_1,\eta_2)\in \mathbb{R}^3\times \big(\mathbb{R}^3\setminus 0\big)\;:\;\phi(t_1,t_2)=0\;\text{or}\;\nabla \phi(t_1,t_2)=0\}.
\end{equation*}

\noindent Now suppose that the collection $\Sigma_\alpha=\phi^{-1}(\alpha)$ forms a folliation of a neighborhood of $\Sigma_0$ on $\mathbb{R}^2$ by connected, smooth curves. So we can apply our Theorem \ref{thm:propagation-Gevrey} for this structure $\mathcal{V}$ and for $\{x_0\}\times\Sigma_0$, for every $x_0\in\mathbb{R}$. Note that in this example it is important that $\{x_0\}\times\Sigma_0$ is contained in the base projection of $\mathrm{T}^0$, otherwise the structure would be elliptic on $\{x_0\}\times\Sigma_0$, and we would not need to use our theorem in this case. Now we give some simple examples of such functions $\phi$:

\begin{enumerate}
    \item $\phi(t_1,t_2)=(t_1-1)^2+(t_2-1)^2-2$;\\
    \item $\phi(t_1,t_2)=t_1-t_2$;\\
    \item $\phi(t_1,t_2)=(t_1+1)(t_2+1)-1$.
\end{enumerate}

\section*{Acknowledgements}

I wish to express my gratitude to Prof. Paulo D. Cordaro for his careful guidance during my Ph.D., and I also wish to thank the reserach group at the University of S\~ao Paulo (S\~ao Paulo and S\~ao Carlos), and at the Federal University of S\~ao Carlos for the helpful seminars and conversations, and in especial  Luis F. Ragognette for his careful reading of the preprint. Finally I wish to thank CNPq for the financial support.

\end{document}